\documentclass[10pt]{article}
\usepackage{amssymb,latexsym}
\usepackage{amsmath}
\usepackage{amsthm,amsfonts}
\usepackage{setspace}
\usepackage{xcolor}
\usepackage{xfrac}
\usepackage{cancel}
\usepackage{graphicx,bbm}
\usepackage{caption}
\usepackage{dsfont}
\usepackage{enumitem,hyperref}
\usepackage{mathtools}
\hypersetup{
    unicode=false,          
    pdftoolbar=true,        
    pdfmenubar=true,        
    pdffitwindow=false,     
    pdfstartview={FitH},    
    pdftitle={My title},    
    pdfauthor={Author},     
    pdfsubject={Subject},   
    pdfcreator={Creator},   
    pdfproducer={Producer}, 
    pdfkeywords={keyword1} {key2} {key3}, 
    pdfnewwindow=true,      
    colorlinks=true,       
    linkcolor=blue,          
    citecolor=red,        
    filecolor=magenta,      
    urlcolor=cyan           
}
\usepackage{cleveref}

\newcounter{indicem}
\newcommand{\HH}{\mathcal{H}}
\newcommand{\OO}{\mathcal{O}}

\newcommand{\JJ}{\mathcal{J}}
\newcommand{\LL}{\mathcal{L}}
\newcommand{\PP}{\mathcal{P}}
\newcommand{\FF}{\mathcal{F}}
\newcommand{\QQ}{\mathcal{Q}}
\newcommand{\U}{\mathcal{U}}

\newcommand{\XX}{\mathcal{X}}
\newcommand{\N}{\mathbb{N}}
\newcommand{\E}{\mathbb{E}}

\newcommand{\R}{\mathbb{R}}

\newcommand{\F}{\mathbb{F}}
\newcommand{\Prob}{\mathbb{P}}

\newcommand{\Finn}{\begin{flushright} \rule{2mm}{2mm}\end{flushright}}

\def\OO       {{\cal O}}

\font\dsrom=dsrom10 scaled 1200
\def \1{\textrm{\dsrom{1}}}

\newcommand{\norme}[1]{\left\lVert#1\right\rVert}

\newtheorem{theorem}{Theorem}[section]
\newtheorem{defi}[theorem]{Definition}

\newtheorem{prop}[theorem]{Proposition}
\newtheorem{lemma}[theorem]{Lemma}
\newtheorem{remark}[theorem]{Remark}

\numberwithin{equation}{section}

\makeatletter
\def\namedlabel#1#2{\begingroup
	#2%
	\def\@currentlabel{#2}%
	\phantomsection\label{#1}\endgroup
}
\makeatother

\usepackage[margin=1.0in]{geometry}

\begin{document}



\title{Optimality Conditions for Control Systems Governed by Monotone Stochastic Evolution Equations\thanks{Funded by Normandy Region and the European Union through ERDF research and innovation program, under the grant for "chaire d'excellence COPTI" and by Agence Nationale de la Recherche (ANR), project ANR-22-CE40-0010 COSS. This research was conducted while the third named author was affiliated with INSA Rouen Normandie.}}

\author{Ioana Ciotir\thanks{INSA Rouen Normandie, Normandie Univ, LMI UR 3226, F-76000 Rouen, France. E-mail: {\tt \{ioana.ciotir,nicolas.forcadel,hasnaa.zidani\}@insa-rouen.fr}}\and Nicolas Forcadel\footnotemark[2]\and Piero Visconti\thanks{Institut f\"{u}r Mathematik, Alpen-Adria-Universit\"{a}t Klagenfurt, Universit\"{a}tstra\ss e 65–67, 9020 Klagenfurt, Austria. E-mail: {\tt piero.visconti@aau.at}}\and Hasnaa Zidani\footnotemark[2]}

\date{\today}
\maketitle

\begin{abstract}
We study a class of optimal control problems governed by nonlinear stochastic equations of monotone type under certain coercivity and linear growth conditions. We give first order necessary conditions of optimality. A stochastic Pontryagin principle can be recovered in the case that the diffusion doesn't depend on the control. We give several applications, most notably for stochastic porous media equations in the Lipschitz case.
\end{abstract}




\section{Introduction}

This paper concerns optimal control problems for systems governed by nonlinear stochastic evolution equations of monotone type. We aim to derive necessary optimality conditions in the style of the Pontryagin Maximum Principle. We introduce as a motivating example a class of problems which model diffusion in heterogeneous medium subject to stochastic effects. This example serves to illustrate the type of nonlinear monotone SPDEs we consider. Our results are formulated in a general abstract framework that encompasses a broader class of nonlinear stochastic evolution equations.

Let $\OO \subset \R^d$ be a bounded open domain with smooth boundary, and let $(\Omega, \mathcal{F}_T, \F, \Prob)$ denote a complete and separable filtered probability space where $\F = \{\mathcal{F}_t:0 \leq t \leq T\}$ represents a filtration. Assume that $W$ is a cylindrical Wiener process on $L^2(\OO)$, and $\F$ is the filtration generated by $W$, enlarged by the $\Prob$-null sets.

We consider the following state equation of stochastic porous media type:
\begin{equation}\label{StateEqIntroSPM} 
	\begin{aligned} dx_t &= \left[\Delta\Psi(x_t) + u_t^\OO(\xi)\right]dt + \sigma(x_t)dW_t, & (\xi, t) \in \OO \times (0,T), \\
		\Psi(x_t) &= u_t^\partial(S), & (S, t) \in \partial\OO \times (0,T), \\ 
		x_0 &\in H^{-1}(\OO), & \xi \in \OO, \end{aligned} 
\end{equation} 
where $\Psi : \R \to \R$ is a Lipschitz continuous, monotone function, $u^\OO$ (resp. $u^\partial$) is a progressively measurable $L^2(\OO)$ (resp. $L^2(\partial\OO)$)-valued stochastic process representing a control on the domain (resp. on the boundary). The assumptions on these functions will be made precise in Section \ref{sec:appSPM}.

The study of porous media equations has garnered significant attention in the mathematical literature due to their diverse applications in fields such as fluid dynamics, material science, and environmental modeling. The order of growth of $\Psi$ determines the diffusion behavior, ranging from slow to fast and even super-fast diffusion. Although the porous media equation is a guiding example throughout the paper, the theoretical contributions extend to general stochastic control problems governed by monotone-type nonlinear SPDEs.

In the deterministic setting, a rich body of work has been dedicated to analyzing diffusion equations in porous media, with results addressing the existence, uniqueness, and regularity of solutions, particularly for nonlinear monotone operators (see, for instance, \cite{marinoschi}). The diffusion behavior in such systems is typically determined by the growth properties of the nonlinear term defined by the function $\Psi$. Depending on these properties, the system may exhibit saturation (associated with singular behavior), slow diffusion (super-unitary growth), fast diffusion (sub-unitary growth), or even super-fast diffusion (negative order growth).


In the stochastic setting, recent research has extended these analyses to include additive or multiplicative noise terms, giving rise to an active area of investigation. Studies focus on how noise influences the properties of solutions, such as their regularity and asymptotic behavior. The monograph \cite{PM} provides a comprehensive overview of advancements in this domain. This framework encompasses the Stefan problem for phase transitions, including mushy regions corresponds to choosing a smooth function $\Psi$, which has been studied in the deterministic setting \cite{lacey1983existence, lacey1983mushy, azaiez2016two}, we mention the controllability results of \cite{barbu2021mushy} for this class of problem (in divergence form). The smoothness and coercivity assumptions of our approach exclude the stochastic Stefan problem but encompass systems with smoothed enthalpy.

The goal of our work is to derive necessary optimality conditions for the control problem associated with \eqref{StateEqIntroSPM} and a cost functional of the form:
\begin{equation}\label{costIntro} \JJ(x, u) := \E \int_0^T \int_\OO f(t,x_t, u_t^\OO) d\xi dt + \E \int_0^T \int_{\partial\OO} g(t,u_t^\partial) dS dt + \E \int_\OO h(x_T, \xi) d\xi.
\end{equation}
Such problems have practical relevance, as they model diverse phenomena, including phase transitions and diffusion processes, under uncertainty.

This paper provides an analysis of the stochastic porous media control problem by establishing first-order necessary optimality conditions for an open-loop control setting. The main novelty lies in addressing the challenges posed by the nonlinear principal part in the state equation. 
While the porous media equation serves as a concrete illustration, the formulation of the problem is abstract and accommodates a wide class of nonlinear monotone SPDEs. In this paper, we provide several applications beyond the porous media setting to demonstrate the scope of our framework.

The theory of optimal control for deterministic systems has seen remarkable growth since its inception in the mid-20th century, with significant contributions in both linear and nonlinear settings. Stochastic control theory emerged as a natural extension to account for uncertainties in the system dynamics, leveraging the mathematical foundations of stochastic processes. The control of stochastic partial differential equations (SPDEs) represents a more recent development, addressing complex systems arising in fields like physics, biology, and engineering.

In the deterministic setting, optimal control of infinite dimensional systems is by now a mature field, we refer for instance to \cite{Li_Yong_95,troltzsch2010optimal} for a review of the topic. The optimal control of deterministic porous media equations has been studied, with results on existence, uniqueness, and optimality conditions (see, e.g., \cite{marinoschi} and references therein).

For stochastic systems, more general SPDE control problems have been considered, as in \cite{bensoussan1983stochastic,hu1990maximum,zhou1993necessary,oksendal,guatteri2011stochastic,du2013maximum,fuhrman2013stochastic,lu2014general,RelaxedCtrl2013,Schrod2018,Schrod2020}, but these works cover, at most, lower order semilinearities. Recent advances include first-order optimality conditions for classes of SPDEs with nonlinear dynamics, in \cite{barbu2020optimal} the authors have established first-order conditions for closed-loop control problems including stochastic porous media equations, by a formulation using Kolmogorov equations.


To the best of our knowledge, this work is the first to provide necessary optimality conditions for an open-loop stochastic control problem with a nonlinear principal operator. While we focus on Lipschitz continuous $\Psi$, this framework paves the way for addressing more general settings, possibly including operators with polynomial growth and/or multivalued operators, since the latter cases can be studied as limits of problems with Lipschitz nonlinearities (as in \cite{PM}).
\section{Formulation of the problem. Main results}
\subsection{Notations and Preliminaries}
Let us start by establishing some notations that will be used in this paper. For any measure space $({\cal{M}}, \Sigma; \mu)$, any Banach space $(Y, \| \cdot \|_Y)$, and for $r\in[1,+\infty[$, the Bochner space $L^r({\cal{M}}; Y)$  consists of all (equivalence classes of) strongly measurable functions $y:[0,T]\longrightarrow  Y$ for which the function $t\longmapsto  \|y(t)\|^r_Y$ is $\mu$-integrable and where functions which agree a.e.  are identified.		 
This space is a Banach space when endowed with the norm  $\|\cdot\|_{L^r({\cal M};Y)}$ defined by
$$ 
\|y\|_{L^r({\mathcal{M}};Y)}:= \left(\int_{\mathcal{M}}\|y(s)\|_Y^r\,d\mu(s)\right)^{\frac1r}. $$
Let $T>0$ be a fixed final time,  we shall denote by $C([0,T];Y)$ the Banach space 
that consists of all continuous functions $y:[0,T]\longrightarrow  Y$.  This space is endowed with the usual norm 
$$\|y\|_{C([0,T];Y)}:= \max_{t\in[0,T]} \|y(t)\|_Y.$$

For any Banach spaces $Y$ and $Z$, we denote by ${\cal L}(Y,Z)$ the space of linear continuous operators from $Y$ into $Z$. When $Y = Z$, we will simply write ${\cal L}(Y)$. The dual space of $Y$, denoted $Y'$, is the space of continuous linear functionals from $Y$ to $\mathbb{R}$, i.e., ${\cal L}(Y,\R)$. The duality pairing between $Y$ and $Y'$ is written as $\langle \cdot, \cdot \rangle_{Y,Y'}$.

For any operator $A \in {\cal L}(Y,Z)$, its adjoint, denoted $A^* \in {\cal L}(Z', Y')$, is defined by the relation
\[
\langle y, A^* z' \rangle_{Y,Y'} := \langle A y, z' \rangle_{Z,Z'}\qquad 
\forall y \in Y, \ \forall z'\in Z.
\]

For a topological space $X$ and Banach spaces $Y,Z$, we say that an operator valued function $A:X\to \LL(Y,Z)$ is strongly continuous if $x\mapsto [A(x)]y$ is continuous from $X$ to $Z$ for every $y\in Y$. We say that it is uniformly continuous if it's continuous from $X$ to $\LL(Y,Z)$ with the metric induced by the operator norm. In the particular case that $Z=\R$, we have $\LL(Y,Z)=Y'$, strong and uniform continuity are equivalent to continuity into $Y'$ with the weak-$*$ and norm topologies, respectively. For a function $f:Y\to Z$, we denote its Gateaux derivative at $y$ by $\partial_yf(y)$, we say $f$ is $C^1$ if $\partial_y f$ is strongly continuous.

We also denote by ${\cal L}_2(Y, Z)$ the space of Hilbert-Schmidt operators from $Y$ to $Z$. This space is a Hilbert space itself, equipped with the inner product
\[
\langle S, Q \rangle_{{\cal L}_2(Y,Z)} = \sum_i \langle S e_i, Q e_i \rangle_{Z}
\]
where $\{e_i\} \subset Y$ is an orthonormal basis. Note that we do not systematically identify ${\cal L}_2(Y,Z)$ with its dual space ${\cal L}_2(Y, Z)'$. Instead, we identify ${\cal L}_2(Y,Z)'$ with the space ${\cal L}_2(Y,Z')$, and the duality pairing is given by
\[
\langle S, Q' \rangle_{{\cal L}_2(Y,Z), {\cal L}_2(Y,Z')} = \sum_i \langle S e_i, Q' e_i \rangle_{Z,Z'}.
\]

Throughout this work, we assume that we are working on a complete filtered probability space $(\Omega, \mathcal{F}_T, \mathbb{F}, \mathbb{P})$, where $\mathbb{F} = \{\mathcal{F}_t\}_{0 \leq t \leq T}$ is a filtration. For any $Y$-valued random variable $y : \Omega \rightarrow Y$ that is Bochner integrable, the expectation of $y$, denoted by $\mathbb{E}[y]$, is defined as
\[
\mathbb{E}[y] = \int_\Omega y(\omega) d\mathbb{P}(\omega).
\]
It is worth noting that $\mathbb{E}[y]$ is an element of $Y$ whenever $y$ is Bochner integrable. We systematically omit the dependence of functions on the random parameter $\omega$.

For a Banach space $Y$, we use the notation $L^p_{\mathbb{F}}(\Omega; C([0,T], E))$ (respectively, $L^p_{\mathbb{F}}(\Omega; L^q(0,T; E))$) to denote the space of progressively measurable functions that belong to $L^p(\Omega; C([0,T], E))$ (respectively, $L^p(\Omega; L^q(0,T; E))$). Dependence on $t$ may often be indicated with a subscript.

Finally, let $Y$ be a Banach space and let $K$ be a nonempty subset of $Y$.  For $y\in K$, we define, $T^b_K(y)$, the adjacent cone to $K$ at $y$ (see, for instance, \cite[section 4.1]{aubin2009set}) as follows
\begin{eqnarray}
	d\in T^b_K(y)  & \Longleftrightarrow&  d\in Y,  \quad\forall\{\varepsilon_k\}_{k\in\N}\subset (0,\infty):\varepsilon_k\to 0 \text{ as }k\to\infty \notag\\
	& & \exists \{d_k\}_{k\in\N}\subset X : d_k\to d \text{ as }k\to\infty\text{ and } y+\varepsilon_kd_k\in K \quad\forall k.\notag
\end{eqnarray}

For embedded spaces $V_0\subset H\subset V_1$, we say that they are a Gelfand triple if $(V_0,(\cdot,\cdot)_{V_0})$ is a separable Hilbert \if{reflexive Banach}\fi space densely embedded in the separable Hilbert space $(H,(\cdot,\cdot)_H)$ which is itself densely embedded in the normed space $V_1$, and the function mapping $h\in H$ to the functional $(\cdot,h)$ on $V_0$ extends to an isometry between $V_1$ and the dual space ${V_0}'$. In particular, this means $V_1$ is a separable Hilbert space \if{separable and reflexive\fi as well. We then define the pairing $$\langle y,x\rangle_{V_1,V_0}=\langle x,y\rangle_{V_0,V_1}=\lim_{\overset{y_n\to y}{y_n\in H}}(x,y_n)_H.$$

\subsection{Optimal control problem}

Let $T>0$ be a fixed final time. Let $V_0\subset H\subset V_1$ be a Gelfand triple on which the state will take its values. Suppose that $W$ is a cylindrical Wiener process on some separable Hilbert space $K$ and $\F$ is the filtration generated by $W$, enlarged by the $\Prob$-null sets, we denote $\LL_0:=\LL_2(K,H)$. We denote the state space
\begin{equation*}
\XX:=L_{\F}^2(\Omega;C([0,T],H))\cap L_{\F}^2(\Omega;L^2(0,T;V_0)).
\end{equation*}

Let $U$ be the control space, which is a separable Banach space. We take the set of admissible controls $U^{ad}\subset U$ to be a given closed subset. We denote \begin{equation*}
\U:=L_{\F}^2(\Omega;L^2(0,T;U)),\qquad \U^{ad}:=\{u\in \U: u_t(\omega)\in U^{ad} \text{ a.e. }\Omega\times[0,T]\}.
\end{equation*}

We consider a general abstract stochastic partial differential equation:
\begin{equation}\label{SEESPM}
	\left\{ \begin{array}{ll}
		dx_t=A(t,x_t,u_t)dt+\sigma(t,x_t,u_t)dW_t, & t\in [0,T],\\
		x_0\in H.
	\end{array}\right.
\end{equation}
We make the following assumptions: \begin{enumerate}[label={$(A_{\arabic*})$},labelwidth=1.5cm]
	\item{\label{hyp:dyn:1}} Suppose that the function $A:\Omega\times[0,T]\times V_0\times U\to V_1$ satisfies:\begin{enumerate}[label={$\arabic*.$}]
		\item For every $(x,z,u)\in V_0^2\times U$ the mapping $(\omega,t)\mapsto \langle A(t,x,u),z\rangle_{V_1,V_0}$ is progressively measurable.
		
		\item For every $(\omega,t)\in \Omega\times[0,T]$, the mapping $(x,u)\mapsto A(t,x,u)$ is $C^1$ from $V_0\times U$ to $V_1$. For every $(\omega,t,x,u)\in \Omega\times[0,T]\times V_0\times U$ we have \begin{equation}\label{AbsBoundDrift}
			\norme{\partial_x A(t,x,u)}_{\LL(V_0,V_1)}+\norme{\partial_u A(t,x,u)}_{\LL(U,V_1)}\le B.
		\end{equation} Moreover, $\norme{A(t,0,0)}_{V_1}\in L^2(\Omega\times[0,T])$.
		
		\item For every $(\omega,t,x,z,u)\in \Omega\times[0,T]\times V_0^2\times U$, we have \begin{equation}\label{AbsCoer}
			\langle \partial_xA(t,x,u)z,z\rangle_{V_1,V_0}\le -M_*\norme{z}_{V_0}^2+\alpha\norme{z}_{H}^2,\end{equation} for some $\alpha\in \R$ and $M_*>0$.
		
	\end{enumerate}
	
	\item{\label{hyp:dyn:2}} Suppose that the function $\sigma:\Omega\times[0,T]\times H\times U\to \LL_0$ satisfies:\begin{enumerate}[label={$\arabic*.$}]
		\item For every $(x,z,u,k)\in H^2\times U \times K$ the mapping $(\omega,t)\mapsto \langle \sigma(t,x,u)k,z\rangle_{H}$ is progressively measurable.
		
		\item For every $(\omega,t)\in \Omega\times[0,T]$, the mapping $(x,u)\mapsto \sigma(t,x,u)$ is $C^1$ from $H\times U$ to $\LL_0$. For every $(\omega,t,x,u)\in \Omega\times[0,T]\times V_0\times U$ we have \begin{equation}\label{AbsBoundDiff}
			\norme{\partial_x\sigma(t,x,u)}_{\LL(H,\LL_0)}+\norme{\partial_u\sigma(t,x,u)}_{\LL(U,\LL_0)}\le L.
		\end{equation} Moreover, $\norme{\sigma(t,0,0)}_{\LL_0}\in L^2(\Omega\times[0,T])$.
	\end{enumerate}
\end{enumerate}

In the next section, we will define precisely what we mean as a solution of \eqref{SEESPM} and establish well-posedness. We will also state an existence and uniqueness result for a solution of \eqref{SEESPM}, for every control input  $u\in \U^{ad}$

In the sequel, we consider the cost function $\JJ: \XX\times \U \longrightarrow \R$ defined by 
\begin{equation}\label{cost}
	\JJ(X,u):=\E\left[ \int_{0}^{T}f(t,x_t,u_t)dt+h(x_T)\right].
\end{equation}

This cost is supposed to satisfy the following assumptions. 
\begin{enumerate}[label={$(A_{\arabic*})$}, resume,labelwidth=0.8cm]
	\item{\label{hyp:3:Cost}} The mapping $(\omega,t)\mapsto f(t,x,u)$ is progressively measurable for every $(x,u)\in V_0\times U$.
	
	The mapping $(x,u)\mapsto f(t,x,u)$ belongs to $C^1(V_0\times U,\R)$ for almost every $(\omega,t)\in \Omega\times[0,T]$. Moreover, the derivatives satisfy $$\norme{\partial_{x,u}f(t,x,u)}_{{V_0}'\times{U}'}\le M_f(\omega,t)+C\norme{x}_{V_0}+C\norme{u}_{U},$$ for some $M_f\in L_{\F}^2(\Omega;L^2(0,T))$, and $C\ge 0$.
	Moreover, we assume that $(\omega,t)\mapsto f(t,0,0)\in L_{\F}^1(\Omega\times[0,T])$.
	
	\item{\label{hyp:4:Cost}} The mapping $\omega\mapsto h(x)$ is $\FF_T$-measurable for every $x\in H$.
	
	The mapping $x\mapsto h(x)$ belongs to $C^1(H,\R)$ for almost every $\omega\in \Omega$. Moreover, the derivatives satisfy $$\norme{\partial_{x}h(x)}_{H'}\le M_h(\omega)+C\norme{x}_{H},$$ for some $M_h\in L_{\FF_T}^2(\Omega)$, and $C\ge 0$. Moreover, we assume  that $\omega\mapsto h(0)\in L_{\F}^1(\Omega)$.
	
\end{enumerate}

In this work, are interested in the following optimal control problem:

\begin{equation}\label{ProbPSPM}
	\mbox{Minimize } \left\{\JJ(x,u),\text{ } x\in \XX, \text{ } u\in \U^{ad}, \text{ and } (x,u) \ \mbox{satisfies \eqref{SEESPM}}\right\}.\tag{$\mathcal{P}$}
\end{equation}

In the following, we assume that the control problem admits an optimal solution. Our goal is to derive the necessary optimality conditions that this solution must satisfy, based on first-order variations of the objective functional.

In some situations, we are able to obtain a Pontryagin type maximum principle, which is stronger than a first order optimality condition and doesn't involve derivatives with respect to the control. The main trade-off is that $\sigma$ may not depend on $u$ and we need to strengthen some of the conditions on the dynamics and cost functions. So for some results, we will assume the following additional hypothesis.

\begin{enumerate}[label={$(\tilde{A})$},labelwidth=0.8cm]
	\item{\label{hyp:PMP}} Suppose that for all $u_1,u_2\in U^{ad}$, $x\in V_0$ and $(\omega,t)\in\Omega\times[0,T]$, $$\norme{A(t,x,u_1)-A(t,x,u_2)}_{H}\le C,$$ for some uniform constant $C$. The function $\sigma$ does not depend on $u$. Assume as well that the mappings $(x,u)\mapsto (\partial_x A(t,x,u))^*,(\partial_x \sigma(t,x))^*$ are strongly continuous, the derivative $(x,u)\mapsto \partial_{x}f(t,x,u)$ is norm continuous from $V_0\times U$ to ${V_0}'$, as is $x\mapsto \partial_x h(x)$ from $H$ to $H'$.
\end{enumerate}

\subsection{Main result}
In the following sections, we will first focus on the regularity results of the solution to the state equation and the regularity of the mapping that associates the state to the control. 

Next, we will establish the optimality conditions for problem \eqref{ProbPSPM}, which we can already state in the following theorem. Doing so requires the introduction of the costate equation, whose solutions take values in the costate space \begin{equation*}
\PP:=L_{\F}^2(\Omega;C([0,T],H'))\cap L_{\F}^2(\Omega;L^2(0,T;{V_1}')),
\qquad\QQ:=L_{\F}^2(\Omega;L^2(0,T;\LL_0')).
\end{equation*}

\begin{theorem}\label{th:main}
	Assume that Hypotheses \ref{hyp:dyn:1}-\ref{hyp:4:Cost} hold. 
	Let $\bar x \in \XX $ and $\bar u\in \U^{ad}$.
	Assume that $(\bar{x},\bar{u})$ is an optimal pair for problem \eqref{ProbPSPM}. Then, there 
	exists a process $(p,q)\in \PP\times\QQ$, such that
	\begin{subequations}\label{thm:CO}  
		\begin{enumerate}
			\item[-] The process $(p,q)$ is the solution of \begin{equation}\label{BSEESPM:main}
				\begin{split}
					dp_t&=-[(\partial_xA(t,\bar{x}_t,\bar{u}_t))^* p_t+(\partial_x\sigma(t,\bar{x}_t,\bar{u}_t))^*q_t-\partial_x f(t,\bar{x}_t,\bar{u}_t)]dt+q_tdW_t\\
					p_T&=-\partial_x h(x_T)\in H',
				\end{split}
			\end{equation}
			\if{\begin{equation}\label{CoStateEqSPM}\begin{array}{rll}
					dp_t(\xi)&=-\left[\Psi'(x_t(\xi))\Delta p_t(\xi)+\sum_{k=1}^{\infty} \mu_k q^k_t(\xi)e_k(\xi)-\frac{\partial f}{\partial x}(x_t,\xi,t)\right]dt&\\
					&\qquad+\sum_{k=1}^{\infty} q^k_t(\xi)dw^{k}_{t},&(\xi,t)\in\OO\times(0,T)\\
					p_t(\xi)&=0,&(\xi,t)\in\partial\OO\times(0,T)\\
					p_T(\xi)&=-g(\xi),&\xi\in\OO
				\end{array}
			\end{equation}}\fi
			\item[-] For a.e. $(\omega,t)\in \Omega\times (0,T)$, it holds that \begin{equation}\label{optCondDualPointwise}
				\langle v_t,(\partial_u A(t,\bar{x}_t,\bar{u}_t))^*p_t+(\partial_u \sigma(t,\bar{x}_t,\bar{u}_t))^*q_t-\partial_u f(t,\bar{x}_t,\bar{u}_t)\rangle_{U,U'}\le 0\quad\forall v\in T^b_{U^{ad}}(\bar{u}_t).
			\end{equation}
			\item[-] Assume \ref{hyp:PMP} holds, then for a.e. $(\omega,t)\in \Omega\times (0,T)$, it holds that \begin{equation}\label{optCondPMP}
				\langle A(t,\bar{x}_t,\bar{u}_t),p_t\rangle_{V_1,{V_1}'}-f(t,\bar{x}_t,\bar{u}_t)\ge \langle A(t,\bar{x}_t,u),p_t\rangle_{V_1,{V_1}'}-f(t,\bar{x}_t,u)\quad\forall u\in U^{ad}.
			\end{equation}
			\if{\begin{equation}
				\label{optCondDualPointwiseInterior}
				v^\OO[p_t(\xi)-\partial_uf(\bar X_t(\xi),\bar u^{\OO}_t(\xi),t)]\le0\quad\forall v^{\OO}\in T^b_{U^{\OO}(\xi)}(\bar{u}^{\OO}_t).
			\end{equation} 
			\item[-] For a.e. $(\omega,t,S)\in \Omega\times(0,T)\times\partial\OO$ it holds that
			\begin{equation}
				\label{optCondDualPointwiseBoundary}
				v^\partial\left[\frac{\partial p_t}{\partial\nu}(S)-\partial_uh(\bar u^{\partial}_t(S),t)\right]\le0\quad\forall v^{\partial}\in T^b_{U^{\partial}}(\bar{u}^{\partial}_t(S)).
			\end{equation}}\fi
		\end{enumerate}
	\end{subequations}
\end{theorem}

\begin{remark}
	Condition \eqref{optCondPMP} holds without differenciability with respect to $u$, or even without the vector space structure of $U$, it is sufficient to have controls in a Polish space and continuous cost and dynamics w.r.t. $u$.
\end{remark}

\subsection{Comments on the main result.} 
Let us note that for the model problem \eqref{StateEqIntroSPM}, the optimality conditions have been established in the literature in the case where $\Psi' \equiv 1$, which corresponds to the stochastic heat equation. We refer to \cite{hasnaa1, Li_Yong_95, lasiecka2000control} to cite some of the works addressing the deterministic case, which corresponds to the scenario where all the parameters $\mu_k$ in the state equation are zero. The stochastic problem has also been analyzed in \cite{bensoussan1983stochastic,hu1990maximum,zhou1993necessary,guatteri2011stochastic,lu2014general,PV}, again for the case where $\Psi' \equiv 1$.
Moreover, it should be noted that in the deterministic case with $\Psi' \neq 1$, the optimality conditions have been studied in the monograph \cite{marinoschi}.

\section{State Equation and Adjoint Equation}

In this section we study both the forward and backward equations introduced previously. We provide results on the sensitivity of trajectories with respect to variations in the control. We will use throughout the variational framework of \cite{KR81}. We refer the textbook \cite{prevot2007concise}, in particular, we invite the reader to have in mind the standard well-posedness result \cite[Theorem 4.2.4]{prevot2007concise} and the It\^o formula \cite[Theorem 4.2.5]{prevot2007concise}.

\subsection{State equation}

We consider solutions of the state equation \eqref{SEESPM} in the following sense.

\begin{defi}\label{DefSol}
	We say that $x\in \XX$ is a solution of \eqref{SEESPM} if $A(\cdot,x_{(\cdot)},u_{(\cdot)})\in L_{\F}^2(\Omega;L^2(0,T;V_1))$ and $x$ satisfies \begin{equation}\label{eqDefSol}	x_t=x_0+\int_{0}^{t}\left[A(s,x_s,u_s)ds+\sigma(s,x_s,u_s)dW_s\right],\quad \forall t\in[0,T],\text{ a.s.}
	\end{equation} in $V_1$.
\end{defi}

\begin{prop}\label{propWP}
	For every $x_0\in H$ and $u\in \U$, equation \eqref{SEESPM} admits a unique solution. Moreover, letting $c_1$ be the constant for the Burkholder-Davis-Gundy inequality with $p=1$, the following bound holds \begin{align}
		&\frac{1}{2}\E\sup_{t\in[0,T]}\norme{x_t}_{H}^2+M_*\E\int_{0}^{T}\norme{x_t}_{V_0}^2dt\notag\\\le& e^{2(\alpha+L^2(1+2c_1^2))T}\left(\norme{x_0}_{H}^2+2\E\int_{0}^{T}\left[\frac{1}{M_*}\norme{A(t,0,u_t)}_{V_1}^2+(L^2(1+2c_1^2))\norme{\sigma(t,0,u_t)}_{\LL_0}^2\right]dt\right).
	\end{align}

    Moreover, the control to state mapping is Lipschitz, more precisely: Let $x$, $x'$ be solutions of \eqref{SEESPM} with data $u$, $u'\in \U$ respectively. Then \begin{equation}\begin{split}
			&\frac{1}{2}\E\sup_{t\in[0,T]}\norme{x_t-x_t'}_{H}^2+M_*\E\int_{0}^{T}\norme{x_t-x_t'}_{V_0}^2dt\\&\qquad\le e^{2(\alpha+L^2(1+2c_1^2))T}\left(\frac{B^2}{M_*}+2L^2(1+2c_1^2)\right)\E\int_{0}^{T}\norme{u_t-u_t'}_{U}dt.
		\end{split}
	\end{equation}
\end{prop}

\begin{proof}
	Existence and uniqueness is proved in \cite[Theorem 4.2.4]{prevot2007concise}, we verify the remaining hypotheses:
	
	As a consequence of \eqref{AbsBoundDrift} (resp. \eqref{AbsBoundDiff}), $A$ (resp. $\sigma$) is $B$(resp. $L$)-Lipschitz continuous.
	
	Similarly, because of \eqref{AbsCoer} we have \begin{equation*}
		\langle A(t, x, u),x\rangle_{V_1,V_0}
		\le -\frac{M_*}{2}\norme{x}_{V_0}^2+\alpha\norme{x}_{H}^2+\frac{1}{M_*}\left[B^2\norme{u}_{U}^2+\norme{A(t,0,0)}_{V_1}^2\right].
	\end{equation*}
	
	Lastly, we have weak monotonicity for $A$: \begin{equation}
		\langle A(t,x_1,u)-A(t,x_2,u),x_1-x_2\rangle_{V_1,V_0}
		\label{AbsMonotone}\le -M_*\norme{x_1-x_2}_{V_0}^2+\alpha\norme{x_1-x_2}_{H}^2.
	\end{equation}
	
	For the estimate we use the It\^o formula \cite[Theorem 4.2.5]{prevot2007concise}, 
	\if{
	\begin{equation*}
		\norme{x_t}_{H}^2=\norme{x_0}_{H}^2+\int_{0}^{t}\left[2\langle A(s,x_s,u_s),x_s\rangle_{V_0,V_1}+\norme{\sigma(s,x_s,u_s)}_{\LL_0}^2\right]ds+2\int_{0}^{t}\langle \sigma(s,x_s,u_s)dW_s,x_s\rangle_{H},
	\end{equation*} with which \begin{align*}
		\sup_{s\in[0,t]}\norme{x_s}_{H}^2-&2\int_{0}^{t}\langle A(s,x_s,u_s)-A(s,0,u_s),x_s\rangle_{V_0,V_1}ds\le\norme{x_0}_{H}^2+2\sup_{s\in[0,t]}\left|\int_{0}^{t}\langle \sigma(s,x_s,u_s)dW_s,x_s\rangle_{H}\right|\\+&\int_{0}^{t}\left[\frac{1}{M_*}\norme{A(s,0,u_s)}_{V_1}^2+M_*\norme{x_s}_{V_0}^2+\norme{\sigma(s,x_s,u_s)}_{\LL_0}^2\right]ds,
	\end{align*} applying now}\fi the coercivity inequality \eqref{AbsMonotone}, the Burkholder-Davis-Gundy inequality \if{\begin{align*}
		&\E\sup_{s\in[0,t]}\norme{x_s}_{H}^2+M_*\E\int_{0}^{t}\norme{x_s}_{V_0}^2ds
		\le \norme{x_0}_{H}^2+2c_1\E\left[\left(\int_{0}^{t}\norme{\sigma(s,x_s,u_s)}_{\LL_0}^2\norme{x_s}_{H}^2ds\right)^{1/2}\right]\\
		&\qquad+\E\int_{0}^{t}\left[2\alpha\norme{x_s}_{H}^2+\frac{2}{M_*}\norme{A(s,0,u_s)}_{V_1}^2+2L^2\norme{x_s}_{H}^2+2L^2\norme{\sigma(s,0,u_s)}_{\LL_0}^2\right]ds\\
		&\quad\le\norme{x_0}_{H}^2+2c_1\E\left[\sup_{s\in[0,t]}\norme{x_s}_{H}\left(\int_{0}^{t} \!\!\norme{\sigma(s,x_s,u_s)}_{\LL_0}^2ds\right)^{1/2}\right]\\
		&\qquad+2\E\int_{0}^{t}\left[(\alpha+L^2)\norme{x_s}_{H}^2+\frac{1}{M_*}\norme{A(s,0,u_s)}_{V_1}^2+L^2\norme{\sigma(s,0,u_s)}_{\LL_0}^2\right]ds\\
		&\quad\le\norme{x_0}_{H}^2+\frac{1}{2}\E\sup_{s\in[0,t]}\norme{x_s}_{H}^2+2c_1^2\E\int_{0}^{t} \norme{\sigma(s,x_s,u_s)}_{\LL_0}^2ds\\
		&\qquad+2\E\int_{0}^{t}\left[(\alpha+L^2)\norme{x_s}_{H}^2+\frac{1}{M_*}\norme{A(s,0,u_s)}_{V_1}^2+L^2\norme{\sigma(s,0,u_s)}_{\LL_0}^2\right]ds,
	\end{align*} so that \begin{equation*}
	\begin{split}
		\frac{1}{2}\E\sup_{s\in[0,t]}\norme{x_s}_{H}^2+M_*\E\int_{0}^{t}\norme{x_s}_{V_0}^2ds&\le\norme{x_0}_{H}^2+2\E\int_{0}^{t}\left[(\alpha+L^2(1+2c_1^2))\norme{x_s}_{H}^2\right.\\&\qquad\left.+\frac{1}{M_*}\norme{A(s,0,u_s)}_{V_1}^2+(L^2(1+2c_1^2))\norme{\sigma(s,0,u_s)}_{\LL_0}^2\right]ds.
	\end{split}
\end{equation*}

With this, the result follows from}\fi and the Gr\"onwall Lemma.

\if{We are now in a position to establish}\fi The proof of the Lipschitz continuous dependance of the state with respect to the control is similar.

\end{proof}

	

\if{\begin{proof}
	Follows as well from the It\^o formula \cite[Theorem 4.2.5]{prevot2007concise}:
\begin{align*}
		\norme{x_t-x_t'}_{H}^2=&2\int_{0}^{t}\left[\langle A(s,x_s,u_s)-A(s,x_s',u_s'),x_s-x_s'\rangle_{V_1,V_0}\right]ds\\& +\int_{0}^{t}\norme{\sigma(s,x_s,u_s)-\sigma(s,x_s',u_s')}_{\LL_0}^2ds
    +2\int_{0}^{t}\langle (\sigma(s,x_s,u_s)-\sigma(s,x_s',u_s'))dW_s,x_s-x_s'\rangle_{H}\\
		=&2\int_{0}^{t}\left[\langle A(s,x_s,u_s)-A(s,x_s',u_s),x_s-x_s'\rangle_{V_1,V_0}\right]ds+\int_{0}^{t}\norme{\sigma(s,x_s,u_s)-\sigma(s,x_s',u_s')}_{\LL_0}^2ds
\\&+2\int_{0}^{t}\left[\langle A(s,x_s',u_s)-A(s,x_s',u_s'),x_s-x_s'\rangle_{V_1,V_0}\right]ds\\&+2\int_{0}^{t}\langle (\sigma(s,x_s,u_s)-\sigma(s,x_s',u_s'))dW_s,x_s-x_s'\rangle_{H},
	\end{align*} with which \begin{align*}
		\sup_{s\in[0,t]}\norme{x_s-x_s'}_{H}^2+2M_*\int_{0}^{t}\norme{x_s-x_s'}_{V_0}^2ds \le&2\int_{0}^{t}\left[(\alpha+L^2)\norme{x_s-x_s'}_{H}^2+L^2\norme{u_s-u_s'}_{U}\right]ds\\&+\int_{0}^{t}\left[\frac{B^2}{M_*}\norme{u_s-u_s'}_{U}^2+M_*\norme{x_s-x_s'}_{V_0}^2\right]ds\\&+2\sup_{s\in[0,t]}\left|\int_{0}^{t}\langle (\sigma(s,x_s,u_s)-\sigma(s,x_s',u_s'))dW_s,x_s-x_s'\rangle_{H}\right|,
	\end{align*} 
	applying now the Burkholder-Davis-Gundy inequality 
	\begin{align*}
		&\E\sup_{s\in[0,t]}\norme{x_s-x_s'}_{H}^2+M_*\E\int_{0}^{t}\norme{x_s-x_s'}_{V_0}^2ds\\\le&\E\int_{0}^{t}\left[2(\alpha+L^2)\norme{x_s-x_s'}_{H}^2+\left(\frac{B^2}{M_*}+2L^2\right)\norme{u_s-u_s'}_{U}\right]ds\\&\qquad+2c_1\E\left[\left(\int_{0}^{t}\norme{\sigma(s,x_s,u_s)-\sigma(s,x_s',u_s')}_{\LL_0}^2\norme{x_s-x_s'}_{H^{-1}}^2ds\right)^{1/2}\right]\\\le&\E\int_{0}^{t}\left[2(\alpha+L^2)\norme{x_s-x_s'}_{H}^2+\left(\frac{B^2}{M_*}+2L^2\right)\norme{u_s-u_s'}_{U}\right]ds\\&\qquad+2c_1\E\left[\sup_{s\in[0,t]}\norme{x_s-x_s'}_{H}\left(\int_{0}^{t} \norme{\sigma(s,x_s,u_s)-\sigma(s,x_s',u_s')}_{\LL_0}^2ds\right)^{1/2}\right]\\\le&\E\int_{0}^{t}\left[2(\alpha+L^2)\norme{x_s-x_s'}_{H}^2+\left(\frac{B^2}{M_*}+2L^2\right)\norme{u_s-u_s'}_{U}\right]ds\\&\qquad+\frac{1}{2}\E\sup_{s\in[0,t]}\norme{x_s-x_s'}_{H}^2+2c_1^2\E\int_{0}^{t} \norme{\sigma(s,x_s,u_s)-\sigma(s,x_s',u_s')}_{\LL_0}^2ds,
	\end{align*} so that $$\frac{1}{2}\E\sup_{s\in[0,t]}\norme{x_s-x_s'}_{H}^2+M_*\E\int_{0}^{t}\norme{x_s-x_s'}_{V_0}^2ds$$ $$\le\E\int_{0}^{t}\left[2(\alpha+L^2(1+2c_1^2))\norme{x_s-x_s'}_{H}^2+\left(\frac{B^2}{M_*}+2L^2(1+2c_1^2)\right)\norme{u_s-u_s'}_{U}\right]ds.$$
    
    With this, the result follows from the Gr\"onwall Lemma.
\end{proof}}\fi

\subsection{Linearized state equation}

\subsubsection{Linear perturbation}

We now fix $u,v\in \U$ and consider the notation $x^u$ for the solution process of Equation \eqref{SEESPM} with $u$ as control. We will now show that $x^{u+\varepsilon v}=x^u+\varepsilon z^{u,v}+o(\varepsilon)$ where $z^{u,v}$ solves a linear equation with $v$ as its data. We introduce the linearized equation \begin{equation}\label{LnrSEESPM}
	\begin{split}
		dz^{u,v}_t&=\left[\partial_xA(t,x^u_t,u_t)z^{u,v}_t+\partial_uA(t,x^u_t,u_t)v_t\right]dt\\&\qquad+\left[\partial_x \sigma(t,x^u_t,u_t)z^{u,v}_t+\partial_u \sigma(t,x^u_t,u_t)v_t\right]dW_t\\
		z^{u,v}_0&=0.
	\end{split}
\end{equation} 


Consider a solution of \eqref{LnrSEESPM} as a process $z^{u,v}\in \XX$ for which \begin{equation}\begin{array}{rl}
		z^{u,v}_t=&\displaystyle\int_{0}^{t}\left[(\partial_xA(s,x^u_s,u_s)z^{u,v}_s+\partial_uA(s,x^u_s,u_s)v_s)ds\right.\\&\qquad\left.+(\partial_x \sigma(s,x^u_s,u_s)z^{u,v}_s+\partial_u \sigma(s,x^u_s,u_s)v_s)dW_s\right]
	\end{array}
	,\quad \forall t\in[0,T],\text{ a.s.}
\end{equation} holds in $V_1$.

Following standard applications of It\^o formula, Burkholder-Davis-Gundy inequality and Gr\"onwall lemma we have.

\begin{prop}\label{propWPLnr}
	Equation \eqref{LnrSEESPM} admits a unique solution. Moreover, letting $c_1$ be the constant for the Burkholder-Davis-Gundy inequality with $p=1$, the following bound holds \begin{equation}
		\frac{1}{2}\E\sup_{t\in[0,T]}\norme{z^{u,v}_t}_{H}^2+M_*\E\int_{0}^{T}\norme{z^{u,v}_t}_{V_0}^2dt\le e^{2(\alpha+L^2(1+2c_1^2))T}\E\int_{0}^{T}\left(\frac{B^2}{M_*}+2L^2(1+2c_1^2)\right)\norme{v_t}_{U}dt.
	\end{equation}

    Moreover, the control to state map $U\ni u\mapsto x^u\in\XX$ is $C^1$ and its derivative is given by $\partial_u[u\mapsto x^u](u).v=z^{u,v}\in\XX$ where $z^{u,v}$ is the solution of Equation \eqref{LnrSEESPM}.
\end{prop}

\begin{proof}
	Existence and uniqueness is again direct from \cite[Theorem 4.2.4]{prevot2007concise}. The estimate follows from the It\^o formula, the Burkholder-Davis-Gundy inequality and the Gr\"onwall lemma. \if{For the estimate we use the It\^o formula \cite[Theorem 4.2.5]{prevot2007concise}:
	
	\begin{align*}
		\norme{z^{u,v}_t}_{H}^2=&\int_{0}^{t}\left[2\langle \partial_xA(s,x^u_s,u_s)z^{u,v}_s,z^{u,v}_s\rangle_{V_1,V_0}+\norme{\partial_x \sigma(s,x^u_s,u_s)z^{u,v}_s+\partial_u \sigma(s,x^u_s,u_s)v_s}_{\LL_0}^2\right]ds\\&+2\int_{0}^{t}\langle \left[\partial_x \sigma(s,x^u_s,u_s)z^{u,v}_s+\partial_u \sigma(s,x^u_s,u_s)v_s\right]dW_s,z^{u,v}_s\rangle_{H},
	\end{align*} with which \begin{align*}
		&\sup_{s\in[0,t]}\norme{z^{u,v}_s}_{H}^2-2\int_{0}^{t}\langle \partial_xA(s,x^u_s,u_s)z^{u,v}_s,z^{u,v}_s\rangle_{V_1,V_0}ds\\&\quad\le\int_{0}^{t}\left[2\langle \partial_uA(s,x^u_s,u_s)v_s,z^{u,v}_s\rangle_{V_1,V_0}+\norme{\partial_x \sigma(s,x^u_s,u_s)z^{u,v}_s+\partial_u \sigma(s,x^u_s,u_s)v_s}_{\LL_0}^2\right]ds\\&\qquad+2\sup_{s\in[0,t]}\left|\int_{0}^{t}\langle \left[\partial_x \sigma(s,x^u_s,u_s)z^{u,v}_s+\partial_u \sigma(s,x^u_s,u_s)v_s\right]dW_s,z^{u,v}_s\rangle_{H}\right|,
	\end{align*} applying now the Burkholder-Davis-Gundy inequality \begin{align*}
		\E\sup_{s\in[0,t]}\norme{z^{u,v}_s}_{H}^2+&M_*\E\int_{0}^{t}\norme{z^{u,v}_s}_{V_0}^2ds\le\E\int_{0}^{t}\left[2(\alpha+L^2)\norme{z^{u,v}_s}_{H}^2 +\left(\frac{B^2}{M_*}+2L^2\right)\norme{v_s}_{U}\right]ds\\&+2c_1\E\left[\left(\int_{0}^{t}\norme{\partial_x \sigma(s,x^u_s,u_s)z^{u,v}_s+\partial_u \sigma(s,x^u_s,u_s)v_s}_{\LL_0}^2\norme{z^{u,v}_s}_{H}^2ds\right)^{1/2}\right]\\\le&\E\int_{0}^{t}\left[2(\alpha+L^2)\norme{z^{u,v}_s}_{H}^2 +\left(\frac{B^2}{M_*}+2L^2\right)\norme{v_s}_{U}\right]ds\\&+2c_1\E\left[\sup_{s\in[0,t]}\norme{z^{u,v}_s}_{H}\left(\int_{0}^{t} \norme{\partial_x \sigma(s,x^u_s,u_s)z^{u,v}_s+\partial_u \sigma(s,x^u_s,u_s)v_s}_{\LL_0}^2ds\right)^{1/2}\right]\\\le&\E\int_{0}^{t}\left[2(\alpha+L^2)\norme{z^{u,v}_s}_{H}^2 +\left(\frac{B^2}{M_*}+2L^2\right)\norme{v_s}_{U}\right]ds\\&+\frac{1}{2}\E\sup_{s\in[0,t]}\norme{z^{u,v}_s}_{H}^2+2c_1^2\E\int_{0}^{t} \norme{\partial_x \sigma(s,x^u_s,u_s)z^{u,v}_s+\partial_u \sigma(s,x^u_s,u_s)v_s}_{\LL_0}^2ds,
	\end{align*} so that $$\frac{1}{2}\E\sup_{s\in[0,t]}\norme{z^{u,v}_s}_{H}^2+M_*\E\int_{0}^{t}\norme{z^{u,v}_s}_{V_0}^2ds$$ $$\le\E\int_{0}^{t}\left[2(\alpha+L^2(1+2c_1^2))\norme{z^{u,v}_s}_{H}^2+\left(\frac{B^2}{M_*}+2L^2(1+2c_1^2)\right)\norme{v_s}_{U}\right]ds.$$
    
    With this, the result follows once again from the Gr\"onwall lemma.}\fi



	Let $u,v\in \U$, $x$, $x^{u+\varepsilon v}$ and $z^{u,v}$ be as before. We begin by defining for every $\varepsilon$ the $\LL(V,V^*)$-valued stochastic processes $A_{t,x}^\varepsilon$, $A_{t,u}^\varepsilon$, $\sigma_{t,x}^\varepsilon$ and $\sigma_{t,u}^\varepsilon$ defined by $$A_{t,x}^\varepsilon z=\int_0^1\partial_x A(t,(1-\theta)x_t^u+\theta x_t^{u+\varepsilon v},u_t+\theta \varepsilon v_t)zd\theta\quad z\in V_0,$$ $$A_{t,u}^\varepsilon v=\int_0^1\partial_u A(t,(1-\theta)x_t^u+\theta x_t^{u+\varepsilon v},u_t+\theta \varepsilon v_t)vd\theta\quad v\in U,$$ $$\sigma_{t,x}^\varepsilon z=\int_0^1\partial_x \sigma(t,(1-\theta)x_t^u+\theta x_t^{u+\varepsilon v},u_t+\theta \varepsilon v_t)zd\theta\quad z\in V_0,$$ $$\sigma_{t,u}^\varepsilon v=\int_0^1\partial_u \sigma(t,(1-\theta)x_t^u+\theta x_t^{u+\varepsilon v},u_t+\theta \varepsilon v_t)vd\theta\quad v\in U.$$
	
	Notice that $\norme{A_{t,x}^\varepsilon}_{\LL(V_0,V_1)}\le B$ and $\langle A_{t,x}^\varepsilon z,z\rangle_{V_1,V_0}\le-M_*\norme{z}_{V_0}^2+\alpha\norme{z}_{H}^2$ a.e. on $\Omega\times[0,T]$. Similarly, $\norme{A_{t,u}^\varepsilon}_{\LL(U,V_1)}\le B$, $\norme{\sigma_{t,x}^\varepsilon}_{\LL(H,\LL_0)}\le 
	L$ and $\norme{\sigma_{t,u}^\varepsilon}_{\LL(U,\LL_0)}\le L$.
	
	\if{We denote $r^\varepsilon=x^{u+\varepsilon v} -x^u-\varepsilon z^{u,v}$, so that, again by the It\^o formula, it follows that \begin{align*}
		\norme{r^\varepsilon_t}_{H}^2=&2\int_{0}^{t}\left[\langle A_{s,x}^\varepsilon r^\varepsilon_s,r^\varepsilon_s\rangle_{V_1,V_0}+\varepsilon\langle (A_{s,x}^\varepsilon-\partial_x A(s,x^{u}_s,u_s)) z^{u,v}_s+(A_{s,u}^\varepsilon-\partial_u A(s,x^{u}_s,u_s))v_s,r^\varepsilon_s\rangle_{V_1,V_0}\right]ds\\&+\int_{0}^{t}\norme{\sigma_{s,x}^\varepsilon r^\varepsilon_s+\varepsilon((\sigma_{s,x}^\varepsilon-\partial_x \sigma(s,x^{u}_s,u_s)) z^{u,v}_s+(\sigma_{s,u}^\varepsilon-\partial_u \sigma(s,x^{u}_s,u_s))v_s)}_{\LL_0}^2ds\\&+2\int_{0}^{t}\left\langle\left[\sigma_{s,x}^\varepsilon r^\varepsilon_s+\varepsilon((\sigma_{s,x}^\varepsilon-\partial_x \sigma(s,x^{u}_s,u_s)) z^{u,v}_s+(\sigma_{s,u}^\varepsilon-\partial_u \sigma(s,x^{u}_s,u_s))v_s)\right]dW_s,r^\varepsilon_s\right\rangle_{H},
	\end{align*}}\fi
	
	We denote $r^\varepsilon=x^{u+\varepsilon v} -x^u-\varepsilon z^{u,v}$, so that, with similar arguments as before we find that $$\frac{1}{2}\E\sup_{s\in[0,t]}\norme{r^\varepsilon_s}_{H}^2+M_*\E\int_{0}^{t}\norme{r^\varepsilon_s}_{V_0}^2ds\le2\E\int_{0}^{t}(\alpha+L^2(1+2c_1^2))\norme{r^\varepsilon_s}_{H}^2ds$$ $$+\frac{\varepsilon^2}{M_*}\E\int_{0}^{t}\norme{(A_{s,x}^\varepsilon-\partial_x A(s,x^{u}_s,u_s)) z^{u,v}_s+(A_{s,u}^\varepsilon-\partial_u A(s,x^{u}_s,u_s))v_s}_{V_1}^2ds$$
	$$+2(1+2c_1^2)\varepsilon^2\E\int_{0}^{t}\norme{(\sigma_{s,x}^\varepsilon-\partial_x \sigma(s,x^{u}_s,u_s)) z^{u,v}_s+(\sigma_{s,u}^\varepsilon-\partial_u \sigma(s,x^{u}_s,u_s))v_s}_{\LL_0}^2ds.$$
	
	It follows from the Gr\"onwall lemma and the dominated convergence theorem that \begin{equation}
		\frac{1}{2}\E\sup_{t\in[0,T]}\norme{x^{u+\varepsilon v}_t -x^u_t-\varepsilon z^{u,v}_t}_{H}^2+M_*\E\int_{0}^{T}\norme{x^{u+\varepsilon v}_t -x^u_t-\varepsilon z^{u,v}_t}_{V_0}^2dt=o(\varepsilon^2),
	\end{equation} which is in essence the Gateaux differentiability. We have only left to prove that for fixed $v\in\U$, the derivative $\U\ni u\mapsto \partial_u[u\mapsto x^u](u).v=z^{u,v}\in\XX$ is continuous.
	
	Let $u,u',v\in\U$. Denote $\zeta=z^{u,v}-z^{u',v}$. By the It\^o formula, we have \begin{align*}
		\norme{\zeta_t}_{H}^2=&2\int_{0}^{t}\left[\langle \partial_xA(s,x^{u'}_s,u_s') \zeta_s,\zeta_s\rangle_{V_1,V_0}+\langle (\partial_x A(s,x^{u}_s,u_s)-\partial_x A(s,x^{u'}_s,u_s')) z^{u,v}_s,\zeta_s\rangle_{V_1,V_0}\right]ds\\
		&+2\int_{0}^{t} \langle (\partial_u A(s,x^{u}_s,u_s)-\partial_u A(s,x^{u'}_s,u_s')) v_s,\zeta_s\rangle_{V_1,V_0}ds\\
		&+\int_{0}^{t}\norme{\partial_x \sigma(s,x^{u}_s,u_s)z^{u,v}_s-\partial_x \sigma(s,x^{u'}_s,u_s')z^{u',v}_s+(\partial_u \sigma(s,x^{u}_s,u_s)-\partial_u \sigma(s,x^{u'}_s,u_s'))v_s}_{\LL_0}^2ds\\&+2\int_{0}^{t}\left\langle\left[\partial_x \sigma(s,x^{u}_s,u_s)z^{u,v}_s-\partial_x \sigma(s,x^{u'}_s,u_s')z^{u',v}_s+(\partial_u \sigma(s,x^{u}_s,u_s)-\partial_u \sigma(s,x^{u'}_s,u_s'))v_s\right]dW_s,\zeta_s\right\rangle_{H},
	\end{align*}
	
	With similar arguments as before we find that $$\frac{1}{2}\E\sup_{s\in[0,t]}\norme{\zeta_s}_{H}^2+M_*\E\int_{0}^{t}\norme{\zeta_s}_{V_0}^2ds\le2\E\int_{0}^{t}(\alpha+L^2(1+2c_1^2))\norme{\zeta_s}_{H}^2ds$$ $$+\frac{1}{M_*}\E\int_{0}^{t}\norme{(\partial_x A(s,x^{u}_s,u_s)-\partial_x A(s,x^{u'}_s,u_s')) z^{u,v}_s+(\partial_u A(s,x^{u}_s,u_s)-\partial_u A(s,x^{u'}_s,u_s')) v_s}_{V_1}^2ds$$
	$$+2(1+2c_1^2)\varepsilon^2\E\int_{0}^{t}\norme{(\partial_x \sigma(s,x^{u}_s,u_s)-\partial_x \sigma(s,x^{u'}_s,u_s')) z^{u,v}_s+(\partial_u \sigma(s,x^{u}_s,u_s)-\partial_u \sigma(s,x^{u'}_s,u_s')) v_s}_{\LL_0}^2ds.$$
	
	The result follows once again from the Gr\"onwall lemma and the dominated convergence theorem.
\end{proof}

\subsubsection{Spike perturbation}

In order to establish \eqref{optCondPMP}, we will need to consider a diferent kind of variation for the control and trajectories. We now consider that the assumption \ref{hyp:PMP} holds.

Let $u,v\in \U^{ad}$ and $t_0\in [0,T)$ be given, we introduce a different notion of a perturbation of $u$ which is in some sense moved "toward" $v$ by a "step" of size $\varepsilon$. Instead of a linear perturbation, we consider for $\varepsilon\in(0,T-t_0]$ the so called spike or needle perturbation \begin{equation}\label{DefSpike}
	u_t^\varepsilon:=\begin{cases}
		v_t & t\in [t_0,t_0+\varepsilon),\\u_t & t\notin [t_0,t_0+\varepsilon).
	\end{cases}
\end{equation}

We introduce the following linearized equation \begin{equation}\label{LnrSEE:Spike}
	\begin{split}
		dz^{\varepsilon}_t&=\left[\partial_xA(t,x^u_t,u_t)z^{\varepsilon}_t+(A(t,x^u_t,u^\varepsilon_t)-A(t,x^u_t,u_t))\right]dt\\&\qquad+\partial_x \sigma(t,x^u_t,u_t)z^{\varepsilon}_tdW_t\\
		z^{\varepsilon}_0&=0.
	\end{split}
\end{equation}

We can then describe the variations of $x$ in terms of the variations of the control by the following proposition.

\begin{prop}\label{VarSpike} The following asymptotics for the variations of the state hold as $\varepsilon\to 0$.\begin{enumerate}
		\item $\norme{x^{u^\varepsilon}-x^{u}}_{\XX}=O(\varepsilon)$.
	
		\item $\norme{z^{\varepsilon}}_{\XX}=O(\varepsilon)$.
	
		\item $\frac{1}{\varepsilon}[x^{u^\varepsilon}-x^{u}-z^{\varepsilon}]\rightharpoonup 0$ in $\XX$.
	\end{enumerate}
	
\end{prop}

\begin{proof}
	We denote $$A_{t,x}^\varepsilon z=\int_0^1\partial_x A(t,(1-\theta)x_t^u+\theta x_t^{u^\varepsilon},u_t^\varepsilon)zd\theta\quad z\in V_0,$$ $$\sigma_{t,x}^\varepsilon z=\int_0^1\partial_x \sigma(t,(1-\theta)x_t^u+\theta x_t^{u^\varepsilon})zd\theta\quad z\in V_0,$$
	\begin{enumerate}
		\item We remark that \begin{equation}
			\begin{split}
				d(x^{u^\varepsilon}_t-x^{u}_t)=&\left[A(t,x^{u^\varepsilon}_t,u^{\varepsilon}_t)-A(t,x^u_t,u_t)\right]dt+\left[\sigma(t,x^{u^\varepsilon}_t)-\sigma(t,x^u_t)\right]dW_t\\
				x^{u^\varepsilon}_0-x^{u}_0=&0.
			\end{split}
		\end{equation}
		Since \begin{equation*}
			\begin{split}
				A(t,x^{u^\varepsilon}_t,u^{\varepsilon}_t)-A(t,x^u_t,u_t)=&A(t,x^{u^\varepsilon}_t,u^{\varepsilon}_t)-A(t,x^u_t,u^{\varepsilon}_t)+A(t,x^u_t,u^{\varepsilon}_t)-A(t,x^u_t,u_t)\\
				=&A_{t,x}^\varepsilon(x^{u^\varepsilon}_t-x^{u}_t)+A(t,x^u_t,u^{\varepsilon}_t)-A(t,x^u_t,u_t)
			\end{split}
		\end{equation*} and \begin{equation*}
		\sigma(t,x^{u^\varepsilon}_t)-\sigma(t,x^u_t)=\sigma_{t,x}^\varepsilon(x^{u^\varepsilon}_t-x^{u}_t),
	\end{equation*} we may show through the It\^o formula, the BDG inequality and the Young inequality that \begin{equation*}
		\norme{x^{u^\varepsilon}-x^{u}}_{\XX}^2\le C\E\left( \int_{0}^{T}\norme{A(t,x^u_t,u^{\varepsilon}_t)-A(t,x^u_t,u_t)}_{H}dt\right)^2=O(\varepsilon^2).
	\end{equation*}
	\item Proof is similar to the previous case.
	
	\item We define $r^\varepsilon=x^{u^\varepsilon}-x^{u}-z^{\varepsilon}$, it then holds that \begin{equation*}
		\begin{split}
			dr^\varepsilon_t=&\left[A_{t,x}^\varepsilon r^\varepsilon_t +[A_{t,x}^\varepsilon-\partial_xA(t,x^u_t,u_t)]z^\varepsilon_t\right]dt+\left[\sigma_{t,x}^\varepsilon r^\varepsilon_t+[\sigma_{t,x}^\varepsilon-\partial_x\sigma(t,x^u_t,u_t)]z^\varepsilon_t\right]dW_t\\
			r^\varepsilon_0=&0.
		\end{split}
	\end{equation*}

	It follows that $\norme{r^\varepsilon}_{\XX}\le C\norme{z^\varepsilon}_{\XX}=O(\varepsilon)$, now we notice that \begin{equation}\label{LinearVar}
		\begin{split}
			dr^\varepsilon_t=&\left[\partial_xA(t,x^u_t,u_t) r^\varepsilon_t +[A_{t,x}^\varepsilon-\partial_xA(t,x^u_t,u_t)](r^\varepsilon_r+z^\varepsilon_t)\right]dt\\&\qquad+\left[\partial_x\sigma(t,x^u_t,u_t) r^\varepsilon_t+[\sigma_{t,x}^\varepsilon-\partial_x\sigma(t,x^u_t,u_t)](r^\varepsilon_t +z^\varepsilon_t)\right]dW_t\\
			r^\varepsilon_0=&0.
		\end{split}
	\end{equation}

	By the linearity of \eqref{LinearVar}, if \begin{equation}\label{CondConverge}
		\frac{1}{\varepsilon}\left([A_{t,x}^\varepsilon-\partial_xA(t,x^u_t,u_t)](r^\varepsilon_r+z^\varepsilon_t),[\sigma_{t,x}^\varepsilon-\partial_x\sigma(t,x^u_t,u_t)](r^\varepsilon_r+z^\varepsilon_t)\right)\rightharpoonup0
	\end{equation} in $L_{\F}^2(\Omega\times[0,T];V_1\times\LL_0)$, then $\frac{1}{\varepsilon}r^\varepsilon\rightharpoonup0$ in $\XX$ and the result follows. Consider $(\phi,\psi)\in L_{\F}^2(\Omega\times[0,T];{V_1}'\times{\LL_0}')$, then \begin{equation*}
	\frac{1}{\varepsilon}\left\langle [A_{t,x}^\varepsilon-\partial_xA(t,x^u_t,u_t)](r^\varepsilon_r+z^\varepsilon_t),\phi\right\rangle_{V_1,{V_1}'}=\left\langle \frac{r^\varepsilon_r+z^\varepsilon_t}{\varepsilon},[A_{t,x}^\varepsilon-\partial_xA(t,x^u_t,u_t)]^*\phi\right\rangle_{V_0,{V_0}'}
\end{equation*} and \begin{equation*}
\frac{1}{\varepsilon}\left\langle [\sigma_{t,x}^\varepsilon-\partial_x\sigma(t,x^u_t,u_t)](r^\varepsilon_r+z^\varepsilon_t),\psi\right\rangle_{\LL_0,{\LL_0}'}=\left\langle \frac{r^\varepsilon_r+z^\varepsilon_t}{\varepsilon},[\sigma_{t,x}^\varepsilon-\partial_x\sigma(t,x^u_t,u_t)]^*\psi\right\rangle_{H,H'},
\end{equation*} so by the strong continuity of $\partial_x A^*$ and $\partial_x\sigma^*$ and the dominated convergence theorem, \eqref{CondConverge} holds and therefore $\frac{1}{\varepsilon}r^\varepsilon\rightharpoonup0$ in $\XX$ and the result holds.
	\end{enumerate}
\end{proof}

\subsection{Adjoint equation}

Throughout this section we suppose that $x\in\XX$ and $u\in\U$ solving \eqref{SEESPM} are given.

We remark that ${V_1}'\subset H'\subset {V_0}'$ is a Gelfand triple isometric to $V_0\subset H\subset V_1$. For the adjoint equation we consider the former.


To simplify notation, we write $\partial_xA(t,x_t,u_t)$ (resp. $\partial_u A(t,x_t,u_t)$, $\partial_x\sigma(t,x_t,u_t)$, $\partial_u\sigma(t,x_t,u_t)$, $\partial_x f(t,x_t,u_t)$, $\partial_u f(t,x_t,u_t)$) simply as $\partial_xA(t)$ (resp. $\partial_u A(t)$, $\partial_x\sigma(t)$, $\partial_u\sigma(t)$, $\partial_x f(t)$, $\partial_u f(t)$). 

Consider the backward stochastic evolution equation \begin{equation}\label{BSEESPM}
	\begin{split}
		dp_t&=-[(\partial_xA(t))^* p_t+(\partial_x\sigma(t))^*q_t-\partial_x f(t)]dt+q_tdW_t\\
		p_T&=-\partial_x h(x_T)\in H',
	\end{split}
\end{equation} or equivalently, for all $0\le t\le T$, \begin{equation}\label{eqDefSolAdj}
	p_t+\partial_x h(x_T)+\int_{t}^{T}\partial_xf_sds=\int_{t}^{T}\left[(\partial_xA(s))^* p_s+(\partial_x\sigma(s))^*q_s\right]ds
	-\int_{t}^{T}q_sdW_s.\end{equation}

Notice that since $\partial_x\sigma(s)\in\LL(H,\LL_0)$, we have 
$$(\partial_x\sigma(s))^*\in \LL({\LL_0}',H')=\LL(\LL_2(K,H'),H').$$

We search for a solution $(p,q)\in \PP\times \QQ$.
Analogously to Definition \ref{DefSol} we define a solution as follows \begin{defi}\label{DefSolAdj}
	We say that $(p,q)$ is a solution of \eqref{BSEESPM} if $p\in \PP$, $q\in \QQ$ and $(p,q)$ satisfies \eqref{eqDefSolAdj} in ${V_0}'$.
\end{defi}

\begin{lemma}\label{propDualityRelation}
	Equation \eqref{BSEESPM} admits a unique solution $(p,q)\in \PP\times\QQ$. Moreover, the following duality relation holds:
	
	\begin{align*}
		\langle z_T^{u,v},p_T\rangle_{H,H'}=&\int_{0}^{T}\left[\langle z_t^{u,v},\partial_x f(t)\rangle_{V_0,{V_0}'}+\langle \partial_u A(t) v_t,p_t\rangle_{V_1,{V_1}'}+\langle \partial_u \sigma(t) v_t,q_t\rangle_{\LL_0,{\LL_0}'}\right]dt\\&\qquad+\int_{0}^{T}\left[\langle z^{u,v}_t,qdW_t\rangle_{H,H'}+\langle \left[\partial_x\sigma(t)z_t^{u,v}+\partial_u\sigma(t)v_t\right]dW_t,p_t\rangle_{H,H'}\right],
	\end{align*} with which \begin{equation}\label{dualityRelation}\begin{split}
		0=&\E\int_{0}^{T}\left[\langle z_t^{u,v},\partial_x f(t)\rangle_{V_0,{V_0}'}+\langle \partial_u A(t) v_t,p_t\rangle_{V_1,{V_1}'}+\langle \partial_u \sigma(t) v_t,q_t\rangle_{\LL_0,{\LL_0}'}\right]dt\\&\qquad+\E\langle z_T^{u,v},\partial_x h(x_T)\rangle_{H,H'}.\end{split}
	\end{equation}
\end{lemma}

\begin{proof}
	Existence and uniqueness is a consequence of \cite[Theorem 4.1]{hu1991adapted}. The duality relation is simply the It\^o formula (see for instance \cite[Theorem 4.2.5]{prevot2007concise}) and the polarization identity.
\end{proof}

\section{Optimality conditions}
	
Throughout this section, we assume Hypotheses \ref{hyp:dyn:1}-\ref{hyp:4:Cost} hold.
	
	\if{We denote $F^x_t$, $F^u_t$ the stochastic processes defined by $$\langle z,F^x_t\rangle_{V_0,{V_0}'}=\int_{\OO}z(\xi)\partial_xf(x^u_t(\xi),u^{\OO}_t(\xi),\xi,t)d\xi$$ and $$\langle v,F^u_t\rangle_{U,U'}=\int_{\OO}v^{\OO}(\xi)\partial_uf(x^u_t(\xi),u^{\OO}_t(\xi),\xi,t)d\xi+\int_{\partial\OO}v^{\partial}(S)\partial_uh(u^{\OO}_t(S),S,t)dS,$$ where $v=(v^{\OO},v^{\partial})$.
		
	It follows from Hypotheses \ref{hyp:1} that $$\partial_x F\in L_{\F}^2(\Omega;L^2(0,T;{V_0}'))\quad\text{ and }\quad \partial_u F\in L_{\F}^2(\Omega;L^2(0,T;U'))\equiv \U'.$$
	
	It follows from Hypotheses \ref{hyp:1} that \begin{equation}
		\langle z_t,\partial_x F(t)\rangle_{V_0,{V_0}'}+\langle v_t,\partial_u F(t)\rangle_{U,U'}\in L_{\F}^1(\Omega;L^1(0,T))\quad\text{ and }\quad\langle z_T,\partial_x h(x_T)\rangle_{H,H'}\in L_{\FF_T}^1(\Omega).
	\end{equation}}\fi
	
\begin{prop}\label{propDifcost}
	The cost $\JJ$ is $C^1$ w.r.t. $x$ and $u$ and its derivative is given by \begin{equation}\label{DifCost}
		\langle \partial\JJ(x,u),(z,v)\rangle_{\XX'\times\U',\XX\times\U}=\E\int_{0}^{T}\left[\langle z_t,\partial_x f(t)\rangle_{V_0,{V_0}'}+\langle v_t,\partial_u f(t)\rangle_{U,U'}\right]dt+\E \langle z_T,\partial_x h(x_T)\rangle_{H,H'}
	\end{equation}
\end{prop}
	
\begin{proof} Let $(x,u),(z,v)\in \XX\times\U$. To establish \eqref{DifCost}t, it is sufficient to show that every sequence $\{\varepsilon_k\}$ converging to $0$ has a subsequence for which $$\frac{1}{\varepsilon}\left[\JJ(x+\varepsilon z,u+\varepsilon v)-\JJ(x,u)\right]$$ converges to $$\E\int_{0}^{T}\left[\langle z_t,\partial_x f(t)\rangle_{V_0,{V_0}'}+\langle v_t,\partial_u f(t)\rangle_{U,U'}\right]dt\notag+\E \langle z_T,\partial_x h(x_T)\rangle_{H,H'}.$$ 
		
	\if{Since $x+\varepsilon_k z\to x$ and $u^{\OO}+\varepsilon_k v^{\OO}\to u^{\OO}$ in $L^2(\Omega\times[0,T]\times \OO)$ and similarly $u^{\partial}+\varepsilon_k z^{\partial}\to u^{\partial}$ in $L^2(\Omega\times[0,T]\times \partial\OO)$, there is a subsequence for which convergence holds a.e., due to the bounds of Hypothesis \ref{hyp:1}, the dominated convergence theorem implies the result.}\fi
		
We have $x+\varepsilon_k z\to x$ in $L_{\F}^2(\Omega\times[0,T];V_0)$, $x_T+\varepsilon_k z_T\to x_T$ in $L_{\FF_T}^2(\Omega;H)$ and $u+\varepsilon_k v\to u$ in $L_{\F}^2(\Omega\times[0,T];U)$, so there is a subsequence for which convergence holds a.e., due to the bounds of Hypothesis \ref{hyp:3:Cost}, the dominated convergence theorem implies the differenciability.
		
The strong continuity of $\partial\JJ(x,u)$ follows from the strong continuity of the derivatives $\partial_{x,u}f$, $\partial_x h$ and a classical variation of the dominated convergence theorem.
		
\end{proof}
	
By Proposition \ref{propWP}, $x\in\XX$ is uniquely determined by $u\in\U$ through \eqref{SEESPM}, we denote such solution by $x^u$. It follows that problem \eqref{ProbPSPM} is equivalent to \begin{equation}\label{ProbU}
	\min_{u\in\U^{ad}} \JJ(x^u,u)\tag{$P$}.
\end{equation}
	
It will be useful to revover from \eqref{DifCost} the derivative of $\JJ(x^u,u)$ w.r.t. $u$. Because of Propositions \ref{propWPLnr} and \ref{propDifcost}, $\JJ$ and the solution mapping $U\ni u\mapsto x^u\in\XX$ are both $C^1$, thus the chain rule holds:
	
\begin{prop}
	The function $\U\ni u\mapsto \JJ(x^u,u)$ is $C^1$ and its derivative is given by: \begin{equation}\label{difControlCost}\begin{split}
			\langle \partial_u[u\mapsto\JJ(x^u,u)],v\rangle_{\U',\U}=&\E\displaystyle\int_{0}^{T}\left[\langle z^{u,v}_t,\partial_x f(t)\rangle_{V_0,{V_0}'}+\langle v_t,\partial_u f(t)\rangle_{U,U'}\right]dt\\&\qquad+\E \langle z^{u,v}_T,\partial_x h(x_T)\rangle_{H,H'}.
		\end{split}
	\end{equation}
\end{prop}
		
We can now proceed to the proof of the main result.	
	
\begin{proof}[Proof of Theorem \ref{th:main}]
	We now fix $(x,u)=(\bar{x},\bar{u})$ which are optimal for Problem \ref{ProbU}. Recall that by Lemma \ref{propDualityRelation}, for $v\in\U$ it holds that \begin{equation}\label{optDualityRelation}\begin{split}
			0=&\E\int_{0}^{T}\left[\langle z_t^{\bar{u},v},\partial_x f(t)\rangle_{V_0,{V_0}'}+\langle \partial_u A(t) v_t,p_t\rangle_{V_1,{V_1}'}+\langle \partial_u \sigma(t) v_t,q_t\rangle_{\LL_0,{\LL_0}'}\right]dt\\&\qquad+\E\langle z_T^{\bar{u},v},\partial_x h(x_T)\rangle_{H,H'}.
		\end{split}
	\end{equation} 
			
	Let $v\in T^b_{\U^{ad}}(\bar{u})$ and $\varepsilon_k\to0$. By definition, there exist $\{v_k\}\subset \U$ such that $v_k\to v$ and $\bar{u}+\varepsilon_k v_k\in \U^{ad}$. We denote $u_k:=\bar{u}+\varepsilon_k v_k$ and $x_k=x^{u_k}$. Since $u\mapsto\JJ(x^u,u)$ is $C^1$, it follows that it is locally Lipschitz continuous. Let $R$ be its Lipschitz constant in some neighborhood of $\bar{u}$ and let $k$ be large enough that $u_k$ belongs to said neighborhood. We have \begin{align*}				\lim_{k\to \infty}\frac{\JJ(x^{u+\varepsilon_k v},u+\varepsilon_kv)-\JJ(\bar{x},\bar{u})}{\varepsilon_k}=&\lim_{k\to \infty}\frac{\JJ(x^{u+\varepsilon_k v},u+\varepsilon_k v)-\JJ(x_k,u_k)}{\varepsilon_k}\\&\qquad+\lim_{k\to \infty}\frac{\JJ(x_k,u_k)-\JJ(\bar{x},\bar{u})}{\varepsilon_k}\\\ge& -R\frac{\varepsilon_k\norme{v_k-v}_{\U}}{\varepsilon_k}+0=-R\norme{v_k-v}_{\U}.\end{align*}
    
    Taking $k\to \infty$ we obtain \begin{equation}\label{optCondPrimal} \E\int_{0}^{T}\left[\langle z^{\bar{u},v}_t,\partial_x f(t)\rangle_{V}+\langle v_t,\partial_u f(t)\rangle_{U,U'}\right]dt+\E \langle z^{\bar{u},v}_T,\partial_x h(x_T)\rangle_{H,H'}\ge0.\end{equation}
    
    Replacing \eqref{optDualityRelation} in \eqref{optCondPrimal} we get \begin{equation}\E\int_{0}^{T}\langle v_t,(\partial_u A(t))^*p_t+(\partial_u \sigma(t))^*q_t-\partial_u f(t)\rangle_{U,U'}dt\le0.\end{equation}
			
	We will need the following result, which is proven in \cite[Lemma 4.6]{wang2017optimal} for $X$ finite dimensional; and in \cite[Lemma 3.2.]{frankowska2020first}, with the Clarke tangent cone in place of the adjacent cone and with the assumption that $X$ is a separable Hilbert space. The proof is omitted as it doesn't change substantially.
			
    \begin{lemma}[{\cite[Lemma 4.6]{wang2017optimal},\cite[Lemma 3.2.]{frankowska2020first}}]\label{LemmaPointwise}
	   Assume $X$ is a separable Banach space, $(S,\Sigma,\mu)$ is a $\sigma-$finite measure. Denote $$\mathcal{K}=\left\{y\in L^2(S,\mu;X):y(\cdot)\in K\quad\mu-\text{a.e.}\right\}.$$
				
	   Suppose also that $y':S\to X'$ is weak$^*$-measurable, $\norme{y'}_{X'}\in L^{2}(S,\mu)$ and for every $w\in T^b_{\mathcal{K}}(y)$, $$\int_{S}\langle y'(s),w(s)\rangle_{X',X}d\mu(s)\le 0.$$
				
	   Then $\langle y'(s),w\rangle_{X',X}\le0\quad\forall w\in T^b_{K}(y(s))$ for $\mu-$a.e. $s\in S$ \qed
    \end{lemma}
			
	A direct application of Lemma \ref{LemmaPointwise} recovers \eqref{optCondDualPointwise}.
			
	To end the proof, assuming that \ref{hyp:PMP} holds we establish \eqref{optCondPMP}.
			
	Denote $$H(t,x,u,p)=\langle A(t,x,u),p\rangle_{V_1,{V_1}'}-f(t,x,u).$$
			
	Denote as well $$f_{t,x}^\varepsilon =\int_0^1\partial_x f(t,(1-\theta)\bar{x}_t+\theta x_t^{u^\varepsilon},u_t^\varepsilon)d\theta,$$ $$h_{x}^\varepsilon =\int_0^1\partial_x h((1-\theta)\bar{x}_T+\theta x_T^{u^\varepsilon})d\theta,$$
			
	Similarly to \eqref{dualityRelation}, it holds that \begin{equation*}\begin{split}
			0=&\E\int_{0}^{T}\langle z_t^{\varepsilon},\partial_x f(t)\rangle_{V_0,{V_0}'}dt+\E\int_{t_0}^{t_0+\varepsilon}\langle  A(t,\bar{x}_t,v_t)-A(t,\bar{x}_t,\bar{u}_t),p_t\rangle_{V_1,{V_1}'}dt\\&\qquad+\E\langle z_T^{\varepsilon},\partial_x h(x_T)\rangle_{H,H'}.
		\end{split}
		\end{equation*}
		
	This implies the relation
			
	\begin{equation}\label{EqLandau}\begin{split}
		&\frac{1}{\varepsilon}\E\int_{t_0}^{t_0+\varepsilon}[H(t,\bar{x}_t,\bar{u}_t,p_t)-H(t,\bar{x}_t,v_t,p_t)]dt\\=&\frac{1}{\varepsilon}\E\int_{t_0}^{t_0+\varepsilon}\langle  A(t,\bar{x}_t,v_t)-A(t,\bar{x}_t,\bar{u}_t),p_t\rangle_{V_1,{V_1}'}dt
		+\frac{1}{\varepsilon}\E\int_{t_0}^{t_0+\varepsilon}[f(t,\bar{x}_t,v_t)-f(t,\bar{x}_t,\bar{u}_t)]dt\\
		=&-\frac{1}{\varepsilon}\E\int_{0}^{T}\langle z_t^{\varepsilon},\partial_x f(t)\rangle_{V_0,{V_0}'}dt-\frac{1}{\varepsilon}\E\langle z_T^{\varepsilon},\partial_x h(x_T)\rangle_{H,H'}
		+\frac{1}{\varepsilon}\E[h(x^\varepsilon_T)-h(\bar{x}_T)]-\frac{1}{\varepsilon}\E[h(x^\varepsilon_T)-h(\bar{x}_T)]\\
		&\qquad+\frac{1}{\varepsilon}\E\int_{0}^{T}[f(t,\bar{x}_t,u^{\varepsilon}_t)-f(t,x^{\varepsilon}_t,u^\varepsilon_t)]dt
		+\frac{1}{\varepsilon}\E\int_{0}^{T}[f(t,x^{\varepsilon}_t,u^\varepsilon_t)-f(t,\bar{x}_t,\bar{u}_t)]dt\\
		=&\frac{\JJ(x^\varepsilon,u^\varepsilon)-\JJ(\bar{x},\bar{u})}{\varepsilon}+\frac{1}{\varepsilon}\E\int_{0}^{T}\langle r_t^{\varepsilon},f_{t,x}^\varepsilon\rangle_{V_0,{V_0}'}dt
		+\frac{1}{\varepsilon}\E\int_{0}^{T}\langle z_t^{\varepsilon},f_{t,x}^\varepsilon-\partial_x f(t)\rangle_{V_0,{V_0}'}dt\\
		&\qquad+\frac{1}{\varepsilon}\E\langle r_T^{\varepsilon},h_{x}^\varepsilon\rangle_{H,H'}+\frac{1}{\varepsilon}\E\langle z_T^{\varepsilon},h_{x}^\varepsilon-\partial_x h(x_T)\rangle_{H,H'}.
	\end{split}
	\end{equation}
			
	Because of the norm continuity of assumption \ref{hyp:PMP} and the bounds of Proposition \ref{VarSpike}, the dominated convergence theorem shows that the $\liminf$ of the right hand side is positive. Then \eqref{optCondPMP} follows from a standard localization argument, see step 7 in the proof of \cite[Theorem 12.17]{lu2021} for a detailed proof.
\end{proof}
	
\if{\begin{remark}
	We mention that the passage from \eqref{EqLandau} to \eqref{optCondPMP} has become mathematical folklore in the stochastic control community, but the proof is not as trivial as might seem and the assumption that $\Omega$ is a separable probability space is often not stated explicitly. See step 7 in the proof of \cite[Theorem 12.17]{lu2021} for a detailed proof.
\end{remark}}\fi

\section{Applications}\label{sec:appSPM}

We now mention several classes of optimal control problems for nonlinear SPDEs which are included in the previous results. Throughout this section, $\OO\subset \R^d$ is a bounded and open domain with boundary $\partial\OO$ regular enough that the results of \cite{Grieser2002Laplacian} and of \cite{Grisvard85Elliptic} apply, the exact regularity then depends on dimension and the boundary conditions considered.

\subsection{Nonlinear divergence form equations}

Take the state equation \begin{equation}\label{StateEq:ExDiv}\begin{array}{rll}
		dX_t(\xi)  &=  \left[\nabla\cdot\Phi(\nabla X_t(\xi)) +u_t(\xi)\right]dt + \sum_{k=1}^{\infty} \mu_{k}X_t(\xi) e_{k}(\xi)
		dw^{k}_t,&(\xi,t)\in\OO\times(0,T),\\
		X_t(S) &= 0,&(S,t)\in\partial\OO\times(0,T)\\
		X_0(\xi) &= x_0(\xi),&\xi\in\OO,
\end{array}\end{equation} where $X_t$ is the state process, $\nabla$ is the gradient and $\nabla\cdot$ is the divergence operator, $\Phi: \mathbb{R}^d \rightarrow \mathbb{R}^d$ is a nonlinear function, and the initial state $x_0$ belongs to $L^2(\OO)$. Because $\Phi$ depends on the gradient of the state variable, we use $y$ to denote its $\R^d$-valued input, and use $\partial_y$ for derivatives (gradient or Jacobian) with respect to this input.

The terms $\{w^{k}_t\}$ are mutually independent one dimensional a Wiener processes (Brownian motions) and we assume that $\F$ is the filtration generated by $\{w^{k}_t\}_{k\in\N}$, enlarged by the $\Prob$-null sets.

The functions $\{e_{k}\}$ are an orthonormal basis of $L^{2}(\OO)$ consisting of eigenfunctions of the Laplacian with Dirichlet boundary conditions. Denote the eigenvalues by $\{\lambda_{k}\}$.

The  input $u$ is a progressively measurable $L^{2}(\OO)$-valued stochastic process and represents an interior control on the domain.

Here, we assume that the function $\Phi$ and the coefficients $\{\mu_k\}$ satisfy the following assumptions:
\begin{enumerate}[label={$(H_\arabic*)$}, labelwidth=0.8cm]
	\item{\label{hyp:1:ExDiv}} The function $\Phi: \R^d \to \R^d$ is assumed to be Lipschitz continuous, of class $C^1$, and monotone. Specifically, we assume the existence of constants $0 < M_* \leq M^*$ such that for all $y,z\in\R^d$,the Jacobian $\partial_y\Phi(y)\in\R^{d\times d}$ satisfies $$M_*|z|^2 \leq \partial_y\Phi(y)z\cdot z\leq M^*|z|^2.$$
	
	\item{\label{hyp:2:ExDiv}} For the eigenvalues $\{\lambda_{k}\}$ previously mentioned, it holds that \begin{equation}
		\sum_{k=1}^{\infty} \mu_{k}^{2} \lambda_{k}^{\frac{d-1}{2}} < \infty.
	\end{equation}
	
\end{enumerate}

We consider a closed subset of $U^{ad}\subset\R$ and consider the set of admissible controls defined as follows
\begin{equation*}
	\U^{ad}:=\left\{
	u\in 
	\U\mid 
	u_t(\xi)\in U^{ad} \ \mbox{ for a.e. } 
	(\omega,\xi,t)\in \Omega\times\OO\times(0,T), \right\}.
\end{equation*}

Consider the cost function $\JJ: \XX\times \U \longrightarrow \R$ defined by 
\begin{equation}\label{cost:ExDiv}\begin{split}
		\JJ(X,u):=&\E\left[ \int_{0}^{T}\!\!\!\!\int_{\OO}f(t,X_t(\xi),\nabla X_t(\xi),u_t(\xi),\xi)d\xi dt\right.\\&\qquad\left.+\int_{0}^{T}\!\!\!\!\int_{\partial\OO}g(t,X_t(S),S)dS dt+ \int_{\OO}h(X_t(\xi),\xi)d\xi\right].
	\end{split}
\end{equation}

\begin{enumerate}[label={$(H_{\arabic*})$}, resume,labelwidth=0.8cm]
	\item{\label{hyp:3:ExDiv}} The mappings $(\omega,t)\mapsto f(t,x,y,u,\xi),g(t,x,S)$ are progressively measurable for every $(x,y,u)\in\R\times \R^d\times\R$ and almost every $(\xi,S)\in\OO\times\partial\OO$.
	
	The mapping $(x,y,u)\mapsto (f(t,x,y,u,\xi),g(t,x,S))$ is $C^1$ for almost every $(\omega,t,\xi,S)\in \Omega\times[0,T]\times\OO\times\partial\OO$. Moreover, the derivatives satisfy $$\left|\partial_{x,y,u}f(t,x,y,u,\xi)\right|\le M_f(\omega,t,\xi)+C|x|+C|y|+C|u|,\quad\left|\partial_x g(t,x,S)\right|\le M_g(\omega,t,S)+C|x|,$$ for some $M_f\in L_{\F}^2(\Omega;L^2(0,T;L^2(\OO)))$, $M_g\in L_{\F}^2(\Omega;L^2(0,T;L^2(\partial\OO)))$ and $C\ge 0$.
	
	We have $$(\omega,t,\xi)\mapsto f(t,0,0,0,\xi)\in L_{\F}^1(\Omega\times[0,T]\times\OO),\quad(\omega,S,t)\mapsto g(t,0,S)\in L_{\F}^1(\Omega\times[0,T]\times\partial\OO).$$
	
	\item{\label{hyp:4:ExDiv}} For the final time cost, we assume that $\omega\mapsto h(x,\xi)$ is $\FF_T$-measurable for every $x\in\R$ and almost every $\xi\in\OO$.
	
	The mapping $x\mapsto h(x,\xi)$ is $C^1$ for almost every $(\omega,\xi)\in \Omega\times\OO$.
	Moreover, the derivative satisfies $$\left|\partial_x h(x,\xi)\right|\le M_h(\omega,\xi)+C|x|,$$ for some $M_h\in L_{\FF_T}^2(\Omega;L^2(\OO))$ and $C\ge 0$.

\end{enumerate}

We consider the Gelfand triple $$V_0:=H_{0}^{1}(\OO)\subset H:=H^{-1}(\OO)\subset V_1=H^{-1}(\OO)=(H_{0}^{1}(\OO))'.$$

Let $U= L^2(\OO)$. We define the operator $A:V_0\times U\to V_1$ by \begin{equation}\label{DefAExDiv}
	\langle A(x,u),\phi\rangle_{V_1,V_0}:= -\int_{\OO}\nabla\Phi(x(\xi))\cdot\nabla\phi(\xi) d\xi+\int_{\OO}u(\xi)\phi(\xi) d\xi.
\end{equation}

It's clear that $A\in C^1(V_0\times U,V_1)$ with $$\langle \partial_x A(x,u)z,\phi\rangle_{V_1,V_0}:= -\int_{\OO}\partial_y\Phi(x(\xi))\nabla z(\xi)\cdot\nabla\phi(\xi) d\xi$$ and $$\langle \partial_u A(x,u)v,\phi\rangle_{V_1,V_0}= \int_{\OO}v(\xi)\phi(\xi) d\xi.$$

For $\phi\in V_0$, it holds that  \begin{align*}
	\langle \partial_x A(x,u)z+\partial_u A(x,u)v,\phi\rangle_{V_1,V_0}=&-\int_{\OO}\partial_y\Phi(x(\xi))\nabla z(\xi)\cdot\nabla\phi(\xi)d\xi+\int_{\OO}v(\xi)\phi(\xi) d\xi\\\le& \norme{\partial_y\Psi(x(\cdot))\nabla z(\cdot)}_{(L^2)^d}\norme{\nabla \phi}_{(L^2)^d}+\norme{v}_{L^2}\norme{\phi}_{L^2}\\\le& M^*\norme{z}_{V_0}\norme{\phi}_{V_0}+C\norme{v}_U\norme{\phi}_{V_0},
\end{align*} for some $C$ (the constant from the Poincar\'e inequality), so that \eqref{AbsBoundDrift} holds with $B=\max\{M^*,C\}$.

Hypothesis \ref{hyp:1:ExDiv} implies that \begin{equation*}
	\langle \partial_x A(x,u)z,z\rangle_{V_1,V_0}=-\int_{\OO}\partial_y\Phi(x(\xi))\nabla z(\xi)\cdot\nabla z(\xi) d\xi\le -M_*\norme{z}_{V_0}^2
\end{equation*} so that \eqref{AbsCoer} is verified.

We consider the cylindrical Wiener process on $K:=L^2(\OO)$ given by the formal series
\begin{equation*}
	W_t = \sum_{k=1}^{\infty} w^{k}_{t} e_{k}, \quad t \geq 0,
\end{equation*} where $\{w^{k}\}$ is a sequence of independent $1$-dimensional Wiener processes.

We define the multiplicative noise operator $\sigma$ by \begin{equation}\label{DefSigmaEx:Div}
	\left([\sigma(x)]h\right)(\xi)=\sum_{k=1}^{\infty} \mu_{k} \left\langle h, e_{k} \right\rangle_{L^2(\OO)} x(\xi) e_{k}(\xi).
\end{equation}

Based on the bounds established in \cite{Grieser2002Laplacian}, there exists a constant $C>0$ such that
\begin{equation*}
	\sum_{k=1}^{\infty} \mu_{k}^{2} \left| xe_{k} \right|_{H}^{2} \leq C \sum_{k=1}^{\infty} \mu_{k}^{2} \lambda_{k}^{\frac{d-1}{2}} \left| x \right|_{H}^{2}
\end{equation*}
for all $x \in L^2(\OO)$. Consequently, the condition $\sigma \in \LL(H,\LL_0)=\mathcal{L}(L^2(\OO), \mathcal{L}_{2}(L^2(\OO), L^2(\OO)))$ simplifies to 
\begin{equation}\label{hyp:Noise:Div}
	\sum_{k=1}^{\infty} \mu_{k}^{2} \lambda_{k}^{\frac{d+1}{2}} < \infty,
\end{equation} which holds due to assumption \ref{hyp:2:ExDiv}.

Because $\sigma$ depends only on $x$ and belongs to $\LL(H,\LL_0)$, Hypothesis \ref{hyp:dyn:2} is verified trivially.

The costate equation \eqref{BSEESPM:main} then reads as  \begin{equation}\label{CoStateEqSPM:Div}\begin{array}{rll}
		dp_t(\xi)&=-\left[\nabla\cdot\partial_y\Phi(\bar{X}_t(\xi))\nabla p_t(\xi)+\sum_{k=1}^{\infty} \mu_k q^k_t(\xi)e_k(\xi)-\partial_x f(t,\bar{X}_t,\nabla \bar{X}_t,\bar{u}_t,\xi)\right.&\\&\qquad\left.+\nabla\cdot\partial_yf(t,\bar{X}_t,\nabla \bar{X}_t,\bar{u}_t,\xi)\right]dt
		+\sum_{k=1}^{\infty} q^k_t(\xi)dw^{k}_{t},&(\xi,t)\in\OO\times(0,T)\\
		p_t(\xi)&=0,&(\xi,t)\in\partial\OO\times(0,T)\\
		p_T(\xi)&=-\partial_x h(\bar{X}_t,\xi),&\xi\in\OO
	\end{array}
\end{equation} and the maximum principle holds:
\begin{itemize}
	\item[-] For a.e. $(\omega,t,\xi)\in \Omega\times (0,T)\times \OO$, it holds that
	\begin{equation}
		\label{optCondDualPointwiseInterior:Div}
		p_t(\xi)\bar{u}(\xi)-f(t,\bar X_t(\xi),\nabla \bar{X}_t(\xi),\bar u_t(\xi),\xi)\ge p_t(\xi)u-f(t,\bar X_t(\xi),\nabla \bar{X}_t(\xi),u,\xi)\quad\forall u\in U^{ad}.
	\end{equation}
\end{itemize}

\subsection{Burgers type equations}

Take the state equation \begin{equation}\label{StateEq:ExBurg}\begin{array}{rll}
		dX_t(\xi)  &=  \left[\partial_\xi^2 X_t(\xi)+b(X_t(\xi))b(\partial_\xi X_t(\xi))  +u_t(\xi)\right]dt + \sum_{k=1}^{\infty} \mu_{k}X_t(\xi) e_{k}(\xi)
		dw^{k}_t,&(\xi,t)\in(0,1)\times(0,T),\\
		X_t(S) &= 0,&(S,t)\in\{0,1\}\times(0,T)\\
		X_0(\xi) &= x_0(\xi),&\xi\in(0,1),
\end{array}\end{equation} where $X_t$ is the state process, $\partial_\xi$ is the space derivative, $b: \mathbb{R} \rightarrow \mathbb{R}$ is a nonlinear function, and the initial state $x_0$ belongs to $L^2(0,1)$. We consider the same cost structure as the previous example.

\begin{enumerate}[label={$(H_\arabic*)$}, resume, labelwidth=0.8cm]
	\item{\label{hyp:5:ExBurg}} The function $b: \R \to \R$ is assumed to be $C^1$. Suppose that $|b(x)|+|b'(x)|\le C_b$ for all $x\in\R$.
\end{enumerate}

In the case $b(x)=x$ corresponds to the Burgers equation, in this framework we are able to consider only smoothed truncated approximations of the identity, say for instance $b(x)=\lambda\arctan(x/\lambda)$. We remark that one could take independent approximations of the identity for $X$ and $\partial_\xi X$, but we take the same function to simplify.

We consider the same Gelfand triple $$V_0:=H_{0}^{1}(0,1)\subset H:=H^{-1}(0,1)\subset V_1=H^{-1}(0,1)=(H_{0}^{1}(0,1))'.$$

Let $U= L^2(0,1)$. We define the operator $A:V_0\times U\to V_1$ by \begin{equation}\label{DefAExBurg}
	\langle A(x,u),\phi\rangle_{V_1,V_0}:= -\int_{0}^{1}\partial_\xi x(\xi)\cdot\partial_\xi\phi(\xi) d\xi+\int_{0}^{1}\left[b(x(\xi)) b(\partial_\xi x(\xi))\right]\phi(\xi) d\xi+\int_{0}^{1}u(\xi)\phi(\xi) d\xi.
\end{equation}

It's clear that $A\in C^1(V_0\times U,V_1)$ with \begin{equation*}
	\begin{split}
		\langle \partial_x A(x,u)z,\phi\rangle_{V_1,V_0}=& -\int_{0}^{1}\partial_\xi z(\xi)\partial_\xi\phi(\xi) d\xi+\int_{0}^{1}\left[b'(x(\xi))b(\partial_\xi x(\xi))\right]z(\xi)\phi(\xi) d\xi\\&\qquad+\int_{0}^{1}\left[b(x(\xi))b'(\partial_\xi x(\xi))\right]\partial_\xi z(\xi)\phi(\xi) d\xi
	\end{split}
\end{equation*} and \begin{equation*}
	\langle \partial_u A(x,u)v,\phi\rangle_{V_1,V_0}= \int_{0}^{1}v(\xi)\phi(\xi) d\xi.
\end{equation*}

For $\phi\in V_0$, it holds that  \begin{align*}
	\langle \partial_x A(x,u)z+\partial_u A(x,u)v,\phi\rangle_{V_1,V_0}=&-\int_{0}^{1}\partial_\xi z(\xi)\partial_\xi\phi(\xi)d\xi+\int_{0}^{1}\left[b'(x(\xi))b(\partial_\xi x(\xi))\right]z(\xi)\phi(\xi) d\xi\\&\qquad+\int_{0}^{1}\left[b(x(\xi))b'(\partial_\xi x(\xi))\right]\partial_\xi z(\xi)\phi(\xi) d\xi+\int_{0}^{1}v(\xi)\phi(\xi) d\xi\\\le& \norme{z}_{V_0}\norme{\phi}_{V_0}+C_b^2C^2\norme{z}_{V_0}\norme{\phi}_{V_0}\\&\qquad+C_b^2C\norme{z}_{V_0}\norme{\phi}_{V_0}+C\norme{v}_{U}\norme{\phi}_{V_0},
\end{align*} for some $C$ (the constant from the Poincar\'e inequality), so that \eqref{AbsBoundDrift} holds.

Hypothesis \ref{hyp:5:ExBurg} implies that \begin{equation*}\begin{split}
	\langle \partial_x A(x,u)z,z\rangle_{V_1,V_0}=&-\int_{0}^{1}|\partial_\xi z(\xi)|^2d\xi+\int_{0}^{1}\left[b'(x(\xi))b(\partial_\xi x(\xi))\right]|z(\xi)|^2d\xi\\&\qquad+\int_{0}^{1}\left[b(x(\xi))b'(\partial_\xi x(\xi))\right]\partial_\xi z(\xi)z(\xi) d\xi\\\le& -\norme{z}_{V_0}^2+C_b^2\norme{z}_{H}^2+C_b^2\norme{z}_{V_0}\norme{z}_{H}\\\le& -\frac{1}{2}\norme{z}_{V_0}^2+\left(C_b^2+\frac{C_b^4}{2}\right)\norme{z}_{H}^2
\end{split}
\end{equation*} so that \eqref{AbsCoer} is verified.

We borrow the assumptions \ref{hyp:2:ExDiv}-\ref{hyp:4:ExDiv} of the previous example, as well as the formulation of $W$ and $\sigma$.

The costate equation \eqref{BSEESPM:main} then reads as  \begin{equation}\label{CoStateEqSPM:Burg}\begin{array}{rll}
		dp_t(\xi)&=-\left[\partial_\xi^2 p_t(\xi)+b'(\bar{X}(\xi))b(\partial_\xi\bar{X}(\xi))p_t(\xi)\right.&\\&\qquad-\partial_\xi[b(\bar{X}(\xi))b'(\partial_\xi\bar{X}(\xi))p_t(\xi)]+\sum_{k=1}^{\infty} \mu_k q^k_t(\xi)e_k(\xi)&\\&\qquad\left.-\partial_x f(t,\bar{X}_t,\partial_\xi \bar{X}_t,\bar{u}_t,\xi)+\partial_\xi\partial_yf(t,\bar{X}_t,\partial_\xi \bar{X}_t,\bar{u}_t,\xi)\right]dt
		&\\&\qquad+\sum_{k=1}^{\infty} q^k_t(\xi)dw^{k}_{t},&(\xi,t)\in(0,1)\times(0,T)\\
		p_t(\xi)&=0,&(\xi,t)\in\{0,1\}\times(0,T)\\
		p_T(\xi)&=-\partial_x h(\bar{X}_t,\xi),&\xi\in(0,1)
	\end{array}
\end{equation} and the maximum principle holds:
\begin{itemize}
	\item[-] For a.e. $(\omega,t,\xi)\in \Omega\times (0,T)\times (0,1)$, it holds that
	\begin{equation}
		\label{optCondDualPointwiseInterior:Burg}
		p_t(\xi)\bar{u}(\xi)-f(t,\bar X_t(\xi),\nabla \bar{X}_t(\xi),\bar u_t(\xi),\xi)\ge p_t(\xi)u-f(t,\bar X_t(\xi),\nabla \bar{X}_t(\xi),u,\xi)\quad\forall u\in U^{ad}.
	\end{equation}
\end{itemize}

\subsection{Porous media type equations}

Consider the following state equation of the stochastic porous media type:
\begin{equation}\label{StateEq:ExSPM}\begin{array}{rll}
		dX_t(\xi)  &=  \left[\Delta \Psi(X_t(\xi)) +u_t^\OO(\xi)\right]dt + \sum_{k=1}^{\infty} \mu_{k}X_t(\xi) e_{k}(\xi) dw^{k}_t,&(\xi,t)\in\OO\times(0,T),\\
		X_0(\xi) &= x_0(\xi),&\xi\in\OO,
\end{array}\end{equation} 
where $X_t$ is the state process, $\Delta$ is the Laplace operator, for which we will consider different boundary conditions, $\Psi: \mathbb{R} \rightarrow \mathbb{R}$ is a nonlinear function, and the initial state $x_0$ belongs to $L^2(\OO)$.

The functions $\{e_{k}\}$ are an orthonormal basis of $L^{2}(\OO)$ consisting of eigenfunctions of the Laplacian with appropriate boundary conditions. We still denote the eigenvalues by $\{\lambda_{k}\}$.

The  input $u^{\OO}$ is a progressively measurable $L^{2}(\OO)$-valued stochastic processes and represents an interior control on the domain.

The  input $u^{\partial}$ is a progressively measurable $L^{2}(\partial\OO)$-valued stochastic processes and represents a control on the boundary. The way it enters the system will vary with the boundary conditions we will consider.

The operator $\Delta \Psi(X_t)$ captures the diffusion or porous media effects.

The stochastic porous media equation has been the subject of extensive research in various contexts in recent years (see \cite{PM} for a comprehensive monograph on the topic). Depending on the growth rate of the function $\Psi$, the equation models different types of diffusion phenomena. For growth rates above unity, the equation describes slow diffusion processes, whereas for sub-unit growth rates, it corresponds to fast diffusion. In cases where the growth rate is negative, the equation even models super-fast diffusion phenomena.

Here, we assume that the function $\Psi$ and the coefficients $\{\mu_k\}$ satisfy the following assumptions:
\begin{enumerate}[label={$(H_\arabic*)$}, resume, labelwidth=0.8cm]
	\item{\label{hyp:6:ExSPM}} The function $\Psi: \R \to \R$ is assumed to be Lipschitz continuous, of class $C^1$, and monotone. Specifically, we assume the existence of constants $0 < M_* \leq M^*$ such that for all $x\in\R$, it holds that $M_* \leq \Psi'(x)\leq M^*$.
	
	\item{\label{hyp:7:ExSPM}} For the eigenvalues $\{\lambda_{k}\}$ previously mentioned, it holds that \begin{equation}
		\sum_{k=1}^{\infty} \mu_{k}^{2} \lambda_{k}^{\frac{d+1}{2}} < \infty.
	\end{equation}
	
\end{enumerate}

We consider $U^\OO$ and $U^\partial$ two closed subsets of $\R$ and consider the set of admissible controls defined as follows
\begin{eqnarray}
	\U^{ad}:=\left\{
	(u^\OO,u^{\partial})\in 
	\U\mid 
	u^\OO_t(\xi)\in U^\OO \ \mbox{ for a.e. } 
	(\omega,\xi,t)\in \Omega\times\OO\times(0,T), \right.\notag\\
	\left. \mbox{ and }
	u^\partial_t(S)\in U^\partial \ \mbox{ for a.e. } 
	(\omega,S,t)\in \Omega\times\partial \OO\times(0,T) \right\}.\notag
\end{eqnarray}

In the sequel, we consider the cost function $\JJ: \XX\times \U \longrightarrow \R$ defined by 
\begin{equation}\label{cost:ExSPM:D}\begin{split}
		\JJ(X,u):=&\E\left[ \int_{0}^{T}\!\!\!\!\int_{\OO}f(t,X_t(\xi),u^{\OO}_t(\xi),\xi)d\xi dt\right.\\&\qquad\left.+\int_{0}^{T}\!\!\!\!\int_{\partial\OO}g(t,u^{\partial}_t(S),S)dS dt+ \int_{\OO}h(\xi)X_T(\xi)d\xi\right].
	\end{split}
\end{equation}

This cost is decomposed of three parts:  a running cost $f$ in the parabolic cylinder
$\OO\times(0,T)$; a running cost $g$ on the boundary $\partial \OO\times(0,T)$ and a final cost $h$ at time $T$. These three costs are supposed to satisfy the following assumptions. 
\begin{enumerate}[label={$(H_{\arabic*})$}, resume,labelwidth=0.8cm]
	\item{\label{hyp:8:ExSPM}} The mappings $(\omega,t)\mapsto f(t,x,u,\xi),g(t,u,S)$ are progressively measurable for every $(x,u)\in\R^2$ and almost every $(\xi,S)\in\OO\times\partial\OO$.
	
	The mapping $(x,u)\mapsto (f(t,x,u,\xi),g(t,u,S))$ is $C^1$ for almost every $(\omega,t,\xi,S)\in \Omega\times[0,T]\times\OO\times\partial\OO$. Moreover, the derivatives satisfy $$\left|\partial_{x,u}f(t,x,u,\xi)\right|\le M_f(\omega,t,\xi)+C|x|+C|u|,\quad\left|\partial_u g(t,u,S)\right|\le M_g(\omega,t,S)+C|u|,$$ for some $M_f\in L_{\F}^2(\Omega;L^2(0,T;L^2(\OO)))$, $M_g\in L_{\F}^2(\Omega;L^2(0,T;L^2(\partial\OO)))$ and $C\ge 0$.
	
	We have $$(\omega,t,\xi)\mapsto f(t,0,0,\xi)\in L_{\F}^1(\Omega\times[0,T]\times\OO),\quad(\omega,S,t)\mapsto g(t,0,S)\in L_{\F}^1(\Omega\times[0,T]\times\partial\OO).$$
	
	\item{\label{hyp:9:ExSPM}} For the final time cost, we assume that $h\in L_{\mathcal{F}_T}^2(\Omega;H_0^1(\OO))$.

\end{enumerate}

\begin{remark}
	The fact that $g$ doesn't depend on $x$, the linearity of the final time cost and the regularity of $h$ are only imposed because the solution $X$ belongs to $L^2([0,T]\times\Omega\times \OO)$ and is therefore not regular enough for it's traces to be functions on the time and space boundary.
\end{remark}

\subsubsection{Dirichlet boundary conditions}

\if{For Equation \eqref{StateEqSPM}, we use the functional setting of $H^{-1}$ solutions, see for instance \cite[Example 4.1.11]{prevot2007concise}. Which is to say $V_0:=L^2(\OO)$, $H:=H^{-1}(\OO)$ and $V_1=(H^{2}(\OO)\cap H_{0}^{1}(\OO))'$.}\fi

We consider the state equation \eqref{StateEq:ExSPM} with the Dirichlet boundary condition \begin{equation}\label{BCSPM:D}
	\Psi(X_t(S)) = u_t^\partial(S),\qquad(S,t)\in\partial\OO\times(0,T)
\end{equation}

\if{Under \ref{hyp:1:ExSPM}, we will demonstrate that for every control input $(u^{\OO}, u^{\partial})$, the state equation \eqref{StateEqSPM} admits a unique solution  $X\in \XX$.
This result will be  proven in Section \ref{sec:StAdjEqSPM}
\medskip}\fi

We consider the Gelfand triple $$V_0:=L^2(\OO)\subset H:=H^{-1}(\OO)\subset V_1=(H^{2}(\OO)\cap H_{0}^{1}(\OO))'.$$

Let $U= L^2(\OO)\times L^2(\partial\OO)$. We define the operator $A:V_0\times U\to V_1$ by \begin{equation}\label{DefAExSPM:D}
	\langle A(x,u),\phi\rangle_{V_1,{V_1}'}:= \int_{\OO}\Psi(x(\xi))\Delta\phi(\xi) d\xi+\int_{\OO}u^{\OO}(\xi)\phi(\xi) d\xi-\int_{\partial\OO}u^{\partial}(S)\frac{\partial\phi}{\partial\nu}(S)dS
\end{equation} where $\phi\in {V_1}'= H^{2}(\OO)\cap H_{0}^{1}(\OO)$ and $\frac{\partial}{\partial\nu}$ denotes the Neumann trace.

It's clear that $A\in C^1(V_0\times U,V_1)$ with $$\langle \partial_x A(x,u)z,\phi\rangle_{V_1,{V_1}'}:= \int_{\OO}\Psi'(x(\xi))z(\xi)\Delta\phi(\xi) d\xi$$ and for $v=(v^\OO,v^{\partial})$, $$\langle \partial_u A(x,u)v,\phi\rangle_{V_1,{V_1}'}= \int_{\OO}v^{\OO}(\xi)\phi(\xi) d\xi-\int_{\partial\OO}v^{\partial}(S)\frac{\partial\phi}{\partial\nu}(S)dS.$$

For $\phi\in {V_1}'$, it holds that  \begin{align*}
	\langle \partial_x A(x,u)z+\partial_u A(x,u)v,\phi\rangle_{V_1,{V_1}'}=&\int_{\OO}\Psi'(x(\xi))z(\xi)\Delta\phi(\xi) d\xi+\int_{\OO}v^{\OO}(\xi)\phi(\xi) d\xi\\&\qquad-\int_{\partial\OO}v^{\partial}(S)\frac{\partial\phi}{\partial\nu}(S)dS\\\le& \norme{\Psi'(x(\cdot))z(\cdot)}_{V_0}\norme{-\Delta\phi}_{{V_0}'}+C\norme{v}_U\norme{\phi}_{{V_1}'}\\\le& M^*\norme{z}_{V_0}\norme{\phi}_{{V_1}'}+C\norme{v}_U\norme{\phi}_{{V_1}'},
\end{align*} for some $C$, so that \eqref{AbsBoundDrift} holds with $B=\max\{M^*,C\}$.

Hypothesis \ref{hyp:6:ExSPM} implies that \begin{equation*}
	\langle \partial_x A(x,u)z,z\rangle_{V_1,V_0}=\langle \partial_x A(x,u)z,(-\Delta_D)^{-1}z\rangle_{V_{1},{V_1}'}=-\int_{\OO}\Psi'(x(\xi))z^2(\xi) d\xi\le -M_*\norme{z}_{V_0}^2
\end{equation*} where $\Delta_D$ is the Dirichlet Laplacian, so that \eqref{AbsCoer} is verified.

We consider the cylindrical Wiener process on $K:=L^2(\OO)$ given by the formal series
\begin{equation*}
	W_t = \sum_{k=1}^{\infty} w^{k}_{t} e_{k}, \quad t \geq 0,
\end{equation*} where $\{w^{k}\}$ is a sequence of independent $1$-dimensional Wiener processes.

We define the multiplicative noise operator $\sigma$ by \begin{equation}\label{DefSigmaEx:D}
	\left([\sigma(x)]h\right)(\xi)=\sum_{k=1}^{\infty} \mu_{k} \left\langle h, e_{k} \right\rangle_{L^2(\OO)} x(\xi) e_{k}(\xi).
\end{equation}

Based on the bounds established in \cite{Grieser2002Laplacian} and \cite[Chapter 2, Appendix]{PM}, there exists a constant $C>0$ such that
\begin{equation}\label{hyp:Noise:SPM:D}
	\sum_{k=1}^{\infty} \mu_{k}^{2} \left| xe_{k} \right|_{H}^{2} \leq C \sum_{k=1}^{\infty} \mu_{k}^{2} \lambda_{k}^{\frac{d+1}{2}} \left| x \right|_{H}^{2}
\end{equation}
for all $x \in H^{-1}(\mathcal{O})$. Consequently, the condition $\sigma \in \LL(H,\LL_0)=\mathcal{L}(H^{-1}(\OO), \mathcal{L}_{2}(L^2(\OO), H^{-1}(\OO)))$ simplifies to assumption \ref{hyp:7:ExSPM}.

Because $\sigma$ depends only on $x$ and belongs to $\LL(H,\LL_0)$, Hypothesis \ref{hyp:dyn:2} is verified trivially.

The costate equation \eqref{BSEESPM:main} is then set in the Gelfand triple $${V_1}'=H^{2}(\OO)\cap H_{0}^{1}(\OO)\subset H'=H_{0}^{1}(\OO)\subset {V_0}'=L^2(\OO)$$ and reads as  \begin{equation}\label{CoStateEqSPM:D}\begin{array}{rll}
		dp_t(\xi)&=-\left[\Psi'(\bar{X}_t(\xi))\Delta p_t(\xi)+\sum_{k=1}^{\infty} \mu_k q^k_t(\xi)e_k(\xi)-\partial_x f(t,\bar{X}_t,\bar{u}^{\OO}_t,\xi)\right]dt&\\
		&\qquad+\sum_{k=1}^{\infty} q^k_t(\xi)dw^{k}_{t},&(\xi,t)\in\OO\times(0,T)\\
		p_t(S)&=0,&(S,t)\in\partial\OO\times(0,T)\\
		p_T(\xi)&=-h(\xi),&\xi\in\OO
	\end{array}
\end{equation} and applying Lemma \ref{LemmaPointwise} on the product spaces $\Omega\times[0,T]\times\OO$ and $\Omega\times[0,T]\times\partial\OO$ we obtain the optimality conditions:
\begin{enumerate}
	\item[-] For a.e. $(\omega,t,\xi)\in \Omega\times (0,T)\times \OO$, it holds that
	\begin{equation}
		\label{optCondDualPointwiseInterior:D}
		v^\OO[p_t(\xi)-\partial_uf(\bar X_t(\xi),\bar u^{\OO}_t(\xi),t)]\le0\quad\forall v^{\OO}\in T^b_{U^{\OO}(\xi)}(\bar{u}^{\OO}_t).
	\end{equation} 
	\item[-] For a.e. $(\omega,t,S)\in\Omega\times(0,T)\times\partial\OO$ it holds that
	\begin{equation}
		\label{optCondDualPointwiseBoundary:D}
		v^\partial\left[\frac{\partial p_t}{\partial\nu}(S)-\partial_ug(\bar u^{\partial}_t(S),t)\right]\le0\quad\forall v^{\partial}\in T^b_{U^{\partial}}(\bar{u}^{\partial}_t(S)).
	\end{equation}
\end{enumerate}

It's worth noting that these conditions are not merely formal, since $p_t\in H^2(\OO)\cap H_0^1(\OO)$ almost surely, even the Neumann trace is well defined.

\subsubsection{Neumann boundary conditions}

\if{For Equation \eqref{StateEqSPM}, we use the functional setting of $H^{-1}$ solutions, see for instance \cite[Example 4.1.11]{prevot2007concise}. Which is to say $V_0:=L^2(\OO)$, $H:=H^{-1}(\OO)$ and $V_1=(H^{2}(\OO)\cap H_{0}^{1}(\OO))'$.}\fi

We consider the state equation \eqref{StateEq:ExSPM} with the Neumann boundary condition \begin{equation}\label{BCSPM:N}
	\frac{\partial}{\partial\nu}\Psi(X_t(S)) = u_t^\partial(S),\qquad(S,t)\in\partial\OO\times(0,T)
\end{equation}

\if{Under \ref{hyp:1:ExSPM}, we will demonstrate that for every control input $(u^{\OO}, u^{\partial})$, the state equation \eqref{StateEqSPM} admits a unique solution  $X\in \XX$.
	This result will be  proven in Section \ref{sec:StAdjEqSPM}
	\medskip}\fi

For this case, we denote $\HH=\{\varphi\in H^{2}(\OO):\frac{\partial \varphi}{\partial\nu}=0\}$ with the norm $\norme{\varphi}_{\HH}=\norme{-\Delta\varphi+\varphi}_{L^2(\OO)}$. By the weak formulation of the Poisson Equation, $\HH$ is exactly the set of functions $\varphi$ such that the mapping $H^1(\OO)\ni x\mapsto \left( \varphi,x\right)_{H^1(\OO)}$ is $L^2(\OO)$-continuous, this means $\HH$ is the dual of $L^2(\OO)$ with pivot space $H^1(\OO)$. This means the dual of $L^2(\OO)$ with pivot space $(H^{1}(\OO))'$ is $\HH'$.

We consider the Gelfand triple $$V_0:=L^2(\OO)\subset H:=(H^{1}(\OO))'\subset V_1=\HH'.$$

Let $U= L^2(\OO)\times L^2(\partial\OO)$. We define the operator $A:V_0\times U\to V_1$ by \begin{equation}\label{DefAExSPM:N}
	\langle A(x,u),\phi\rangle_{V_1,{V_1}'}:= \int_{\OO}\Psi(x(\xi))\Delta\phi(\xi) d\xi+\int_{\OO}u^{\OO}(\xi)\phi(\xi) d\xi+\int_{\partial\OO}u^{\partial}(S)\phi(S)dS,
\end{equation} where $\varphi\in {V_1}'= \HH$.

It's clear that $A\in C^1(V_0\times U,V_1)$ with $$\langle \partial_x A(x,u)z,\phi\rangle_{V_1,{V_1}'}:= \int_{\OO}\Psi'(x(\xi))z(\xi)\Delta\phi(\xi) d\xi$$ and $$\langle \partial_u A(x,u)v,\phi\rangle_{V_1,{V_1}'}= \int_{\OO}v^{\OO}(\xi)\phi(\xi) d\xi+\int_{\partial\OO}v^{\partial}(S)\phi(S)dS.$$

For $\phi\in {V_1}'$, it holds that  \begin{align*}
	\langle \partial_x A(x,u)z+\partial_u A(x,u)v,\phi\rangle_{V_1,{V_1}'}=&\int_{\OO}\Psi'(x(\xi))z(\xi)\Delta\phi(\xi) d\xi\\&\qquad+\int_{\OO}v^{\OO}(\xi)\phi(\xi) d\xi+\int_{\partial\OO}v^{\partial}(S)\phi(S)dS\\\le& \norme{\Psi'(x(\cdot))z(\cdot)}_{V_0}\norme{-\Delta\phi}_{{V_0}'}+C\norme{v}_U\norme{\phi}_{{V_1}'}\\\le& M^*\norme{z}_{V_0}\left(\norme{\phi}_{{V_1}'}+\norme{\phi}_{{V_0}'}\right)+C\norme{v}_U\norme{\phi}_{{V_1}'},
\end{align*} for some $C$, so that \eqref{AbsBoundDrift} holds.

Hypothesis \ref{hyp:6:ExSPM} implies that \begin{equation*}\begin{split}
	\langle \partial_x A(x,u)z,z\rangle_{V_1,V_0}=&\langle \partial_x A(x,u)z,(-\Delta_N+Id)^{-1}z\rangle_{V_{1},{V_1}'}\\=&-\int_{\OO}\Psi'(x(\xi))z^2(\xi) d\xi+\int_{\OO}\Psi'(x(\xi))z(\xi)\varphi(\xi) d\xi\\\le& -M_*\norme{z}_{V_0}^2+M^*\norme{z}_{V_0}\norme{\varphi}_{{V_0}'}\\\le& -M_*\norme{z}_{V_0}^2+\frac{M_*}{2}\norme{z}_{V_0}^2+\frac{(M^*)^2}{2M_*}\norme{\varphi}_{{V_0}'}^2\\\le& -\frac{M_*}{2}\norme{z}_{V_0}^2+\frac{(M^*)^2}{2M_*}\norme{\varphi}_{H'}^2\\=& -\frac{M_*}{2}\norme{z}_{V_0}^2+\frac{(M^*)^2}{2M_*}\norme{z}_{H}^2,\end{split}
\end{equation*} where $\Delta_N$ is the Neumann Laplacian and $\varphi=(-\Delta_N+Id)^{-1}z$, so that \eqref{AbsCoer} is verified.

For $W$ and $\sigma$, we consider the same situation as the Dirichlet case, with the adaptation that the condition \eqref{hyp:Noise:SPM:D} is replaced by \begin{equation*}
\sum_{k=1}^{\infty} \mu_{k}^{2} \left| xe_{k} \right|_{H}^{2} \leq C \sum_{k=1}^{\infty} \mu_{k}^{2} (1+\lambda_{k})^{\frac{d+1}{2}} \left| x \right|_{H}^{2},
\end{equation*} to account for the first (null) eigenvalue.

The costate equation \eqref{BSEESPM:main} is then set in the Gelfand triple $${V_1}'=\HH\subset H'=H^{1}(\OO)\subset {V_0}'=L^2(\OO)$$ and reads as  \begin{equation}\label{CoStateEqSPM:N}\begin{array}{rll}
		dp_t(\xi)&=-\left[\Psi'(x_t(\xi))\Delta p_t(\xi)+\sum_{k=1}^{\infty} \mu_k q^k_t(\xi)e_k(\xi)-\partial_x f(t,\bar{X}_t,\bar{u}^{\OO}_t,\xi)\right]dt&\\
		&\qquad+\sum_{k=1}^{\infty} q^k_t(\xi)dw^{k}_{t},&(\xi,t)\in\OO\times(0,T)\\
		\frac{\partial p_t}{\partial\nu}(S)&=0,&(S,t)\in\partial\OO\times(0,T)\\
		p_T(\xi)&=-h(\xi),&\xi\in\OO
	\end{array}
\end{equation} and applying Lemma \ref{LemmaPointwise} on the product spaces $\Omega\times[0,T]\times\OO$ and $\Omega\times[0,T]\times\partial\OO$ we obtain the optimality conditions:
\begin{enumerate}
	\item[-] For a.e. $(\omega,t,\xi)\in \Omega\times (0,T)\times \OO$, it holds that
	\begin{equation}
		\label{optCondDualPointwiseInterior:N}
		p_t(\xi)\bar u^{\OO}_t(\xi)-f(\bar X_t(\xi),\bar u^{\OO}_t(\xi),t)\ge p_t(\xi)u-f(\bar X_t(\xi),u,t)\quad\forall u\in U^{\OO}:
	\end{equation} 
	\item[-] For a.e. $(\omega,t,S)\in\Omega\times(0,T)\times\partial\OO$ it holds that
	\begin{equation}
		\label{optCondDualPointwiseBoundary:N}
		 p_t(S)\bar u^{\partial}_t(S)-g(\bar u^{\partial}_t(S),t)\ge p_t(S)u-g(u,t)\quad\forall u\in U^\partial.
	\end{equation}
\end{enumerate}

\subsubsection{Robin boundary conditions}

We consider the state equation \eqref{StateEq:ExSPM} with the Robin boundary condition \begin{equation}\label{BCSPM:R}
	\frac{\partial}{\partial\nu}\Psi(X_t(S))-\alpha(S)\Psi(X_t(S)) = u_t^\partial(S)[\beta+\gamma\dot{w}^\partial_t],\qquad(S,t)\in\partial\OO\times(0,T),
\end{equation} where $\alpha:\partial\OO\to\R$ is a bounded measurable function and $\dot{w}^\partial_t$ represents a one dimensional white noise affecting the boundary condition.

For this case, we borrow from \cite{marinoschi,CiotirRobin24} and endow $H^{1}(\OO)$ with the norm \begin{equation*}
	\int_{\OO}|\nabla\varphi(\xi)|^2d\xi+\int_{\partial\OO}\alpha(S)|\varphi(S)|^2dS,
\end{equation*} which we assume is equivalent to the Sobolev norm (a sufficient conditions for this is that $\alpha$ is nonnegative and bounded away from zero on some open subset of $\partial\OO$, though $\alpha$ could in principle assume negative values). We also assume that $\alpha$ is regular enough that the product $\alpha\varphi$ belongs to $H^{1/2}(\partial\OO)$ for $\varphi\in H^{3/2}(\partial\OO)$ (say, $\alpha\in H^{s}(\partial\OO)$ for $s>\frac{d-2}{2}$, see \cite{MultSobolev21}), due to \cite{Grisvard85Elliptic} the resolvent $(-\Delta_R)^{-1}$ maps $L^2(\OO)$ into $H^2(\OO)$, where $\Delta_R$ is the Robin Laplacian.

We denote $\HH=\{\varphi\in H^{2}(\OO):\frac{\partial \varphi}{\partial\nu}-\alpha\varphi=0\}$ with the norm $\norme{\varphi}_{\HH}=\norme{-\Delta\varphi}_{L^2(\OO)}$. By the weak formulation of the Poisson Equation, $\HH$ is exactly the set of functions $\varphi$ such that the mapping $H^1(\OO)\ni x\mapsto \left( \varphi,x\right)_{H^1(\OO)}$ is $L^2(\OO)$-continuous, this means $\HH$ is the dual of $L^2(\OO)$ with pivot space $H^1(\OO)$. The dual of $L^2(\OO)$ with pivot space $(H^{1}(\OO))'$ is then $\HH'$.

We consider the Gelfand triple $$V_0:=L^2(\OO)\subset H:=(H^{1}(\OO))'\subset V_1=\HH'.$$

Let $U= L^2(\OO)\times L^2(\partial\OO)$. We define the operator $A:V_0\times U\to V_1$ by \begin{equation}\label{DefAExSPM:R}
	\langle A(x,u),\phi\rangle_{V_1,{V_1}'}:= \int_{\OO}\Psi(x(\xi))\Delta\phi(\xi) d\xi+\int_{\OO}u^{\OO}(\xi)\phi(\xi) d\xi+\beta\int_{\partial\OO}u^{\partial}(S)\phi(S)dS,
\end{equation} where $\varphi\in {V_1}'= \HH$.

It's clear that $A\in C^1(V_0\times U,V_1)$ with $$\langle \partial_x A(x,u)z,\phi\rangle_{V_1,{V_1}'}:= \int_{\OO}\Psi'(x(\xi))z(\xi)\Delta\phi(\xi) d\xi$$ and $$\langle \partial_u A(x,u)v,\phi\rangle_{V_1,{V_1}'}= \int_{\OO}v^{\OO}(\xi)\phi(\xi) d\xi+\beta\int_{\partial\OO}v^{\partial}(S)\phi(S)dS.$$

For $\phi\in {V_1}'$, it holds that  \begin{align*}
	\langle \partial_x A(x,u)z+\partial_u A(x,u)v,\phi\rangle_{V_1,{V_1}'}=&\int_{\OO}\Psi'(x(\xi))z(\xi)\Delta\phi(\xi) d\xi\\&\qquad+\int_{\OO}v^{\OO}(\xi)\phi(\xi) d\xi+\beta\int_{\partial\OO}v^{\partial}(S)\phi(S)dS\\\le& \norme{\Psi'(x(\cdot))z(\cdot)}_{V_0}\norme{-\Delta\phi}_{{V_0}'}+C\norme{v}_U\norme{\phi}_{{V_1}'}\\\le& M^*\norme{z}_{V_0}\norme{\phi}_{{V_1}'}+C\norme{v}_U\norme{\phi}_{{V_1}'},
\end{align*} for some $C$, so that \eqref{AbsBoundDrift} holds.

Hypothesis \ref{hyp:6:ExSPM} implies that \begin{equation*}
		\langle \partial_x A(x,u)z,z\rangle_{V_1,V_0}=\langle \partial_x A(x,u)z,(-\Delta_R)^{-1}z\rangle_{V_{1},{V_1}'}=-\int_{\OO}\Psi'(x(\xi))z^2(\xi) d\xi\le -M_*\norme{z}_{V_0}^2
\end{equation*} where $\Delta_R$ is the Robin Laplacian and $\varphi=(-\Delta_R)^{-1}z$, so that \eqref{AbsCoer} is verified.

We consider $K=L^2(\OO)\times \R$ and $W=\left(\sum_{k=1}^{\infty} w^{k}_{t} e_{k},w^{\partial}_t\right)$. Define for $h=(h^\OO,h^\partial)\in K$ and $u=(u^\OO,u^\partial)\in U$ the value of $[\sigma(x,u)]h\in H=(H^1(\OO))'$ by evaluating at $\varphi\in H^1(\OO)$: \begin{equation}\label{DefSigmaEx:R}
	\left\langle[\sigma(x,u)]h,\varphi\right\rangle_{H,H'}=\sum_{k=1}^{\infty} \mu_{k} \left\langle h^\OO, e_{k} \right\rangle_{L^2(\OO)} \int_{\OO}x(\xi) e_{k}(\xi)\varphi(\xi)d\xi+\gamma h^\partial\int_{\partial\OO}u^\partial(S)\varphi(S)dS.
\end{equation}

We verify that $\sigma\in\LL(H\times U,\LL_0)$, so Hypothesis \ref{hyp:dyn:2} is verified trivially.

The costate equation \eqref{BSEESPM:main} is then set in the Gelfand triple $${V_1}'=\HH\subset H'=H^{1}(\OO)\subset {V_0}'=L^2(\OO)$$ and reads as  \begin{equation}\label{CoStateEqSPM:R}\begin{array}{rll}
		dp_t(\xi)&=-\left[\Psi'(x_t(\xi))\Delta p_t(\xi)+\sum_{k=1}^{\infty} \mu_k q^k_t(\xi)e_k(\xi)-\partial_x f(t,\bar{X}_t,\bar{u}^{\OO}_t,\xi)\right]dt&\\
		&\qquad+\sum_{k=1}^{\infty} q^k_t(\xi)dw^{k}_{t}+q^\partial_t(\xi)dw^{\partial}_{t},&(\xi,t)\in\OO\times(0,T)\\
		\frac{\partial p_t}{\partial\nu}(S)&=\alpha(S)p_t(S),&(S,t)\in\partial\OO\times(0,T)\\
		p_T(\xi)&=-h(\xi),&\xi\in\OO
	\end{array}
\end{equation} and applying Lemma \ref{LemmaPointwise} on the product spaces $\Omega\times[0,T]\times\OO$ and $\Omega\times[0,T]\times\partial\OO$ we obtain the optimality conditions:
\begin{enumerate}
	\item[-] For a.e. $(\omega,t,\xi)\in \Omega\times (0,T)\times \OO$, it holds that
	\begin{equation}
		\label{optCondDualPointwiseInterior:R}
		v^\OO[p_t(\xi)-\partial_uf(\bar X_t(\xi),\bar u^{\OO}_t(\xi),t)]\le0\quad\forall v^{\OO}\in T^b_{U^{\OO}(\xi)}(\bar{u}^{\OO}_t).
	\end{equation} 
	\item[-] For a.e. $(\omega,t,S)\in\Omega\times(0,T)\times\partial\OO$ it holds that
	\begin{equation}
		\label{optCondDualPointwiseBoundary:R}
		v^\partial\left[ \beta p_t(S)+\gamma q^\partial_t(S)-\partial_ug(\bar u^{\partial}_t(S),t)\right]\le0\quad\forall v^{\partial}\in T^b_{U^{\partial}}(\bar{u}^{\partial}_t(S)).
	\end{equation}
\end{enumerate}

\begin{remark}
    We remark that with our assumptions on $\Psi$, the state equation \eqref{StateEq:ExSPM} is well posed in the Gelfand triple $H_0^1(\OO)\subset L^2(\OO)\subset H^{-1}(\OO)$ (or $H^1(\OO)\subset L^2(\OO)\subset (H^1(\OO))'$). The operator $A$ (with null Dirichlet boundary conditions) is then \begin{equation*}\begin{split}
        \langle A(x,u),\phi\rangle_{H^{-1}(\OO),H_0^1(\OO)}:=&-\int_{\OO}\nabla\Psi(x(\xi))\cdot\nabla\phi(\xi)d\xi+\int_{\OO}u^\OO(\xi)\phi(\xi)d\xi\\
        =&-\int_{\OO}\Psi'(x(\xi))\nabla x(\xi)\cdot\nabla\phi(\xi)d\xi+\int_{\OO}u^\OO(\xi)\phi(\xi)d\xi.
    \end{split}
    \end{equation*}

    Which means, at least formally, that \begin{equation*}
        \langle \partial_xA(x,u)z,\phi\rangle_{H^{-1}(\OO),H_0^1(\OO)}:=-\int_{\OO}\Psi''(x(\xi))z(\xi)\nabla x(\xi)\cdot\nabla\phi(\xi)d\xi-\int_{\OO}\Psi'(x(\xi))\nabla z(\xi)\cdot\nabla\phi(\xi)d\xi.
    \end{equation*}

    In order to handle the linearized state equation and/or the adjoint equation, the previous operator must be well defined, which (assuming $\Psi''$ exists and is bounded) is only de case in space dimension $d=1$, even in this case the coercivity hypothesis \eqref{AbsCoer} may not hold.
\end{remark}

\bibliographystyle{plain}
\bibliography{biblioProblem}

\end{document}